\documentclass{amsart}
\title{Automorphism Groups and Invariant Theory on PN}

\author{Jo\~ao Alberto de Faria}
\address{
Department of Mathematics\\
Florida Institute of Technology \\
Melbourne, FL}
\email{jdefaria2010@my.fit.edu}

\author{Benjamin Hutz}
\address{
Department of Mathematics and Computer Science\\
Saint Louis University\\
St. Louis, MO}
\email{hutzba@slu.edu}

\thanks{The authors received funding from NSF Grant DMS-1415294. Thanks to Joseph Silverman for pointing out several errors in an earlier version.}



\usepackage{color}
\usepackage{amssymb,amsmath,amsthm,fullpage}
\usepackage[all]{xy}  
\usepackage{url}
\usepackage{rotating} 
\usepackage{tikz}
\usepackage{textcomp,listings}
\usepackage{versions}
\usepackage{algorithmic,algorithm}
\usepackage{caption}

\definecolor{green}{rgb}{0,0.5,0}
\definecolor{dkgreen}{rgb}{0,0.6,0}
\definecolor{gray}{rgb}{0.5,0.5,0.5}
\definecolor{mauve}{rgb}{0.58,0,0.82}
\lstnewenvironment{python}[1][]{
\lstset{
  frame=single,                   
  language=python,                          
  basicstyle=\ttfamily\small, 
  numbers=left,                             
  numberstyle=\scriptsize\color{black},  
  stepnumber=1,                   
  numbersep=7pt,                  
  backgroundcolor=\color{white},      
  stringstyle=\color{red},
  showspaces=false,               
  showstringspaces=false,         
  showtabs=false,                 
  tabsize=2,                      
  captionpos=b,                   
  breaklines=true,                
  breakatwhitespace=false,        
  rulecolor=\color{black},        
  framexleftmargin=1mm, framextopmargin=1mm, frame=shadowbox, 
  alsoletter={1234567890},
  otherkeywords={\ , \}, \{},
  keywordstyle=\color{blue},       
  stringstyle=\color{mauve},         
  emph={access,and,break,class,continue,def,del,elif ,else,%
  except,exec,finally,for,from,global,if,import,in,i s,%
  lambda,not,or,pass,print,raise,return,try,while},
  emphstyle=\color{black}\bfseries,
  emph={[2]True, False, None, self},
  emphstyle=[2]\color{blue},
  emph={[3]from, import, as},
  emphstyle=[3]\color{blue},
  upquote=true,
  morecomment=[s]{"""}{"""},
  commentstyle=\color{dkgreen},       
  emph={[4]1, 2, 3, 4, 5, 6, 7, 8, 9, 0},
  emphstyle=[4]\color{blue},
  literate=*{:}{{\textcolor{blue}:}}{1}%
  {=}{{\textcolor{blue}=}}{1}%
  {-}{{\textcolor{blue}-}}{1}%
  {+}{{\textcolor{blue}+}}{1}%
  {*}{{\textcolor{blue}*}}{1}%
  {!}{{\textcolor{blue}!}}{1}%
  {(}{{\textcolor{blue}(}}{1}%
  {)}{{\textcolor{blue})}}{1}%
  {[}{{\textcolor{blue}[}}{1}%
  {]}{{\textcolor{blue}]}}{1}%
  {<}{{\textcolor{blue}<}}{1}%
  {>}{{\textcolor{blue}>}}{1},%
}}{}

\definecolor{orange}{rgb}{1,0.65,0.17}

\providecommand{\abs}[1]{\left\lvert#1\right\rvert}

 \providecommand{\Spec}[1]{\text{Spec }#1}

\def\Z{\mathbb{Z}}
\def\Q{\mathbb{Q}}

\def\P{\mathbb{P}}
\def\A{\mathbb{A}}

\def\F{\mathbb{F}}

\def\M{\mathcal{M}}

\newcommand{\col}{\,{:}\,}
\newcommand{\tth}{^{\operatorname{th}}}

\DeclareMathOperator{\Gal}{Gal} \DeclareMathOperator{\Hom}{Hom}
 
\DeclareMathOperator{\PGL}{PGL} \DeclareMathOperator{\GL}{GL}

\DeclareMathOperator{\Aut}{Aut}\DeclareMathOperator{\Fix}{Fix}

\DeclareMathOperator{\SL}{SL}\DeclareMathOperator{\ch}{char}

\DeclareMathOperator{\Rat}{Rat}\DeclareMathOperator{\Res}{Res}
 \DeclareMathOperator{\tr}{trace}
 
\DeclareMathOperator{\Char}{char} \DeclareMathOperator{\Proj}{Proj}
 \DeclareMathOperator{\Conj}{Conj}

\DeclareMathOperator{\Ralg}{\textbf{R-Alg}}
\DeclareMathOperator{\Grp}{\textbf{Grp}}
\DeclareMathOperator{\Set}{\textbf{Set}}

\theoremstyle{plain}
\newtheorem{thm}{Theorem}[section]
\newtheorem*{thm*}{Theorem}
\newtheorem{lem}[thm]{Lemma}
\newtheorem{prop}[thm]{Proposition}
\newtheorem{cor}[thm]{Corollary}
\theoremstyle{definition}
\newtheorem{defn}[thm]{Definition}

\newtheorem{exmp}[thm]{Example}
\newtheorem*{exmp*}{Example}

\theoremstyle{remark}
\newtheorem*{rem}{Remark}

\excludeversion{code}

\begin{document}
\maketitle

\begin{abstract}
  Let $K$ be a field and $f:\P^N \to \P^N$ a morphism. There is a natural conjugation action on the space of such morphisms by elements of the projective linear group $\PGL_{N+1}$. The group of automorphisms, or stabilizer group, of a given $f$ for this action is known to be a finite group. In this article, we apply methods of invariant theory to automorphism groups by addressing two mainly computational problems. First, given a finite subgroup of $\PGL_{N+1}$, determine endomorphisms of $\P^N$ with that group as subgroup of its automorphism group. In particular, we show that every finite subgroup occurs infinitely often and discuss some associated rationality problems. Inversely, given an endomorphism, determine its automorphism group. In particular, we extend the Faber-Manes-Viray fixed-point algorithm for $\P^1$ to endomorphisms of $\P^2$. A key component is an explicit bound on the size of the automorphism group depending on the degree of the endomorphism.
\end{abstract}

 \section{Introduction}
    The invariant theory of finite groups is well developed and its connections to complex dynamical systems in dimension one have seen some attention, notably for solving polynomial equations through iteration such as in \cite{Doyle}. However, in higher dimensions less is known and in all dimensions rationality questions such as those in the arithmetic of dynamical systems are mainly unexplored. In this article we address two mainly computational problems. First, given a finite subgroup of $\PGL_{N+1}$, determine endomorphisms of $\P^N$ with that group as subgroup of its automorphism group. In particular, we show that every finite subgroup occurs infinitely often and discuss some associated rationality problems. Inversely, given an endomorphism, determine its automorphism group. To make more precise statements we first need to set-up some terminology.

    Let $K$ be a field and $f:\P^N \to \P^N$ be a morphism of degree $d$. We can represent $f$ as an $N+1$-tuple of homogeneous polynomials of degree $d$, with no common roots. There is a natural conjugation action on the space of such morphisms by elements of the projective linear group $\PGL_{N+1}$:
    \begin{equation*}
        f^{\alpha} = \alpha \circ f \circ \alpha^{-1} \qquad \alpha \in \PGL_{N+1}.
    \end{equation*}
    The quotient by this action is a geometric quotient that we call the moduli space $M_d^N$ of degree $d$ morphisms of projective space \cite{Levy, petsche, Silverman9}. Given a morphism $f: \P^N \to \P^N$, we define the \emph{stabilizer group} or \emph{group of automorphisms}
    \begin{equation*}
        \Aut(f) = \{ \alpha \in \PGL_{N+1} \col f^{\alpha} = f\}.
    \end{equation*}
    The group $\Aut(f)$ is always a finite subgroup of $\PGL_{N+1}$ \cite{petsche}. The existence of a nontrivial automorphism is rare in the sense that there is an open dense set in $M_d^N$ for which the automorphism group is trivial \cite{Levy}. However, the set of morphisms with nontrivial automorphisms is quite interesting, and the additional structure provided by the existing symmetries can lead to stronger results. For example, the set of rational preperiodic points for a family of twists is uniformly bounded \cite{Levy2} and an explicit bound is known over $\Q$ for the family
    \begin{equation*}
        f_{b}(z) = z + \frac{b}{z}.
    \end{equation*}
    Since the existence of a uniform bound even for the family of quadratic polynomials $f_c(z) = z^2+c$ remains an open conjecture, this additional automorphism structure is obviously helpful. The existence of nontrivial automorphisms is also closely related to the field of definition versus field of moduli problem which addresses when a $K$-rational point on $M_d^N$ has a representation as $K$-rational homogeneous polynomials \cite{Hutz11, Silverman12}. In particular, a morphism has a nontrivial automorphism if and only if it has a nontrivial rational twist, i.e. there are two representations of $\xi \in \M_d^N$ that are conjugate over the algebraic closure $\bar{K}$ but not over $K$. Further demonstrating that the maps with nontrivial automorphisms are special, for $d>2$ the singular points of $M_d^1$ are exactly those maps with nontrivial automorphisms \cite{Miasnikov}.

    This leads naturally to questions about determining the automorphism group of a map and the inverse problem: given a group $\Gamma$, determine maps $f$ with $\Gamma \subset \Aut(f)$. In dimension 1, Faber-Manes-Viray developed efficient algorithms to compute the automorphism group of a given morphism \cite{FMV} as well as the more general problem of determining when two given maps are conjugate. In the other direction, for $\P^1$ Silverman gives explicit forms for morphisms whose automorphism groups contain a cyclic or dihedral group. Aside from these explicit forms, the main tool is classical invariant theory.
    As early as Klein \cite{klein}, the connection was known between invariant polynomials of a group $\Gamma$ and morphisms of $\P^1$ that have $\Gamma$ as a subgroup of its automorphism group. These classical results were stated in the language of equivariants, which is equivalent to the notion of automorphisms in the context of this work. In the course of studying solutions to quintic equations, Doyle and McMullen \cite{Doyle} gave the full connection between invariants and automorphisms in dimension 1. Crass has used similar techniques for higher dimensional systems \cite{Crass, Crass2}.
    Miasnikov-Stout-Williams \cite{Miasnikov} partially address the inverse problem in dimension 1, which maps $f$ have $\Gamma \subset \Aut(f)$. They use results from West \cite{West}, which come from applying invariant theory to binary forms to study the dimension of the locus in $M_d^1$ that have a particular finite subgroup of $\PGL_2$ as subgroup of their automorphism group.

    In this article, we first address the inverse problem for $\P^N$. Relying heavily on classical invariant theory, we study which finite subgroups of $\PGL_{N+1}$ can occur as automorphism groups or subgroups of automorphism groups. Specifically, given a representation of a finite subgroup of $\PGL_{N+1}$, we give explicit constructions to determine maps with that group as automorphism group or subgroup of its automorphism group. We also address algorithms for determining the automorphism group of a map in $\P^2$ and generalizing the work of Faber-Manes-Viray \cite{FMV}. The determination of the automorphism group of a given map represents the Master's thesis of the first author and is implemented in the Sage computer algebra system for maps defined over $\Q$ \cite{defaria}.

\section{Results and Outline}
    We begin by reviewing some of the standard results in classical invariant theory concerning the ring of invariants and the module of equivariants in Section \ref{sect.invariant}. In Section \ref{sect.inv_and_auto} (Theorem \ref{thm_form_aut}), we prove the general connection between invariant $(N+1)$-forms and maps with nontrivial automorphism on $\P^N$. This generalizes the statement found in Klein \cite{klein} for $\P^1$ and was stated but not proven by Crass in \cite{Crass2}. It provides one of the main tools for solving the inverse problem since there are many algorithms available for computing invariants \cite{gatermann, Sturmfels2}.
    \begin{thm*}[Theorem \ref{thm_form_aut}]
        Define
        \begin{equation*}
            dX^I = (-1)^{\sigma_I}dx_{i_1} \wedge \cdots \wedge dx_{i_n}
        \end{equation*}
        where $I$ is the ordered set
        \begin{equation*}
            \{i_1,\ldots,i_n\} \quad i_1 < \cdots < i_n
        \end{equation*}
        and for $\hat{i}$ the index not in $I$, $\sigma_I$ is the permutation
        \begin{equation*}
            \begin{pmatrix}
              0 &1&\cdots&n\\
              \hat{i} &i_1&\ldots & i_n
            \end{pmatrix}.
        \end{equation*}
        $\Gamma$ invariant $n$-forms
        \begin{equation*}
            \phi = \sum_{\hat{i}=0}^n f_{\hat{i}}dX^I
        \end{equation*}
        are in 1-1 correspondence with maps
        \begin{equation*}
            f = (f_0,\dots, f_n)
        \end{equation*}
        with $\Gamma \subset \Aut(f)$.
    \end{thm*}
    We then prove in Theorem \ref{thm_exist_aut} that all finite subgroups of $\PGL_{N+1}$ are realized as subgroups of automorphism groups. It was already known that all finite subgroups of $\PGL_2$ are realized \cite{Miasnikov}.
    \begin{thm*}[Theorem \ref{thm_exist_aut}]
        Let $\Gamma$ be a finite subgroup of $\PGL_{N+1}$. Then there are infinitely many morphisms $f:\P^N \to \P^N$ such that $\Gamma \subseteq \Aut(f)$.
    \end{thm*}
    However, questions concerning fields of definition for these maps and their exact automorphism group are more delicate subjects. We show in Theorem \ref{thm_no_tetra} that all subgroups of $\PGL_2$ are realized as automorphism groups (as opposed to just subgroups of automorphism groups) over $\Q$ except for the tetrahedral group which requires $\Q(\sqrt{-3})$. In Section \ref{sect_no_tetra} we also give an example $f$ with $\Aut(f) = \Gamma$ for each finite subgroup $\Gamma$ of $\PGL_2$ (Figure \ref{fig:exact}).
    \begin{thm*}[Theorem \ref{thm_no_tetra}]
        There is no map $f:\P^1 \to \P^1$ defined over $\Q$ which has tetrahedral group as automorphism group.
    \end{thm*}

    In Section \ref{sect.P2}, we give the classification of finite subgroups of $\PGL_3$ in modern notation as Silverman did for $\PGL_2$ \cite{Silverman12}. In Section \ref{sect_bound}, we prove in Theorem \ref{thm_degree_bound} an explicit bound on the size of $\Aut(f)$ in terms of the degree of $f$.  This degree bound is a key piece of extending the algorithm of Faber-Manes-Viray \cite{FMV} to $\P^2$ in Section \ref{sect.alg}.
    \begin{thm*}[Theorem \ref{thm_degree_bound}]
        Let $f:\P^2 \to \P^2$ be a morphism of degree $d \geq 2$. Then
        \begin{equation*}
            \#\Aut(f) \leq 6d^6.
        \end{equation*}
    \end{thm*}
    For $\P^1$ there is the bound $\#\Aut(f) \leq \max(60,2d+2)$ \cite[Example 2.52]{Silverman20}.
    More generally, Levy proved the existence of such a bound for $\P^N$ but did not formulate an explicit (nontrivial) value.  Levy-Manes-Thompson \cite{Levy2} use the existence of such a bound to prove that the number of rational preperiodic points in a family of twists is bounded. Their bound is possible because the size of the automorphism group bounds the degree of the field of definition of any twist. It would be an interesting question to determine an optimal bound.

    In Section \ref{sect.alg}, we switch to the problem of determining the automorphism group of a given morphism. We define the automorphism scheme for $ \P^N$ and prove in Theorem \ref{thm_aut_scheme} that it is a closed finite group scheme as in \cite{FMV} for $\P^1$.
    \begin{thm*}[Theorem \ref{thm_aut_scheme}]
        Let $R$ be a noetherian commutative ring and let $f:\P^N_R \to \P^N_R$ be a morphism of degree at least $2$. Then the functor $\Aut_f$ is represented by a closed finite $R$-subgroup scheme $\Aut(f) \subset \PGL_{N+1}$.
    \end{thm*}
    Using the bound on the size of the automorphism group in terms of the degree from Section \ref{sect_bound} and an analysis of the combinatorics of preperiodic points, we are able to generalize the fixed point algorithm of \cite{FMV} to morphisms of $\P^2$. For base field $\Q$, the algorithm is implemented in the Sage computer algebra system \cite{sage}.

    Section \ref{sect.alg.exmp} is devoted to examples. In Section \ref{sect.exmp_dim2} we use invariant theory to determine maps $f:\P^2 \to \P^2$ with specific groups $\Gamma$ with $\Gamma \subset \Aut(f)$. In Section \ref{sect.FP.exmp} we give examples of computing the automorphism of a given $f:\P^2 \to \P^2$. Also in this section we discuss some of the running time issues of the algorithm in general and, in particular, with our Sage implementation.

\section{Invariant Theory} \label{sect.invariant}
    In this section we summarize results from invariant theory in a way suitable for application to automorphism groups of projective morphisms. The invariant theory of finite groups is well developed and combines elements of combinatorics and computer algebra. The main references for this material are \cite{gatermann,smith2,Sturmfels2,Worfolk}.

    Let $\Gamma$ be a finite group, $K$ a field of coefficients, $V$ a finite dimensional $K$-vector space, and $\rho:\Gamma \to \GL(V)$ a representation of $\Gamma$ on $V$, By means of $\rho$, the group $\Gamma$ acts on $V$ through linear substitutions. Denote $K[V]$ the algebra of polynomial functions, which is the symmetric algebra on the dual $V^{\ast}$. In other words, if $V$ is an $N$ dimensional vector space and $x_1,\ldots,x_N$ is a basis for $V^{\ast}$, then $K[V] = K[x_1,\ldots,x_N]$ which we denote by $K[\bar{x}]$. The group $\Gamma$ acts on $K[V]$ by
    \begin{equation*}
        (\gamma f)(v) = f(\rho(\gamma^{-1})v).
    \end{equation*}

    \begin{defn}
       We say that $F$ is a \emph{(relative) invariant} of $\Gamma$ if for all $\gamma \in \Gamma$, $\gamma F = \chi(\gamma)F$ for some linear group character $\chi$.  The set of all invariants is a ring denoted $K[\bar{x}]^{\Gamma}$. We will denote $K[\bar{x}]^{\Gamma}_{\chi}$ the ring of relative invariants associated to the character $\chi$.
    \end{defn}
    Hilbert proved that for a reducible representation of $\Gamma$, the ring of invariants $K[\bar{x}]^{\Gamma}$ is finitely generated.
    Furthermore, Eagon and Hochester proved that if the order of the group is relatively prime to $\ch{K}$, then $K[\bar{x}]^{\Gamma}$ is Cohen-Macaulay \cite{Eagon}. In particular, there are (algebraically independent) homogeneous invariants $p_1,\ldots, p_N$ and $s_1,\ldots, s_m$, where $m=\frac{1}{\abs{\Gamma}}\prod \deg(p_i)$ such that $K[\bar{x}]^{\Gamma}$ is generated by $s_1,\ldots,s_m$ as a $K[p_1,\ldots,p_n]$ module. The $p_i$ are called \emph{primary invariants} and the $s_i$ are called \emph{secondary invariants}.
     For a finite group in characteristic zero (and in more general situations), there are well-established algorithms to compute the primary and secondary invariants \cite{Kemper,Sturmfels2}.

    The Molien series and the Reynolds (projection) operator are two tools from invariant theory that allow one to compute invariants.
    \begin{defn}
        The \emph{Molien series} for a finite group $\Gamma$ and a linear group character $\chi$ is given by
        \begin{equation*}
            H_{K[\bar{x}]^{\Gamma}}(t) = \frac{1}{\abs{\Gamma}}\sum_{\gamma \in \Gamma} \frac{\chi(\gamma)}{\det(1-t\gamma)}.
        \end{equation*}
        This was shown by Molien (1897) to be the Hilbert-Poincar\'e Series of $K[\bar{x}]^{\Gamma}$ generated by $\chi$-invariants \cite{gatermann}.  That is,
        \begin{equation*}
        \frac{1}{\abs{\Gamma}}\sum_{\gamma \in \Gamma} \frac{\chi(\gamma)}{\det(1-t\gamma)}
        = \sum_{k=0}^{\infty}\dim\left(K[\bar x]_k^{\Gamma}\right) t^k,
        \end{equation*}
        where $K[\bar x]_k^{\Gamma}$ is the space of degree $k$ invariants.

        The \emph{Reynolds operator} for a finite group $\Gamma$ and a linear group character $\chi$ is the projection operator $R_{\chi} : K[\overline{x}] \to K[\overline{x}]^{\Gamma}_{\chi}$ given by
        \begin{align*}
            R_{\chi}(F) &= \frac{1}{\abs{\Gamma}}\sum_{\gamma \in \Gamma} \chi(\gamma) \gamma F.
        \end{align*}
    \end{defn}

    Noether gave an upper bound of $\abs{\Gamma}$ for the degrees of the primary invariants \cite{Noether}. Furthermore \cite[\S3.5]{Derksen}, the degrees, $e_1,\ldots,e_m$, of the secondary invariants are given by
    \begin{equation*}
        t^{e_1} + \cdots + t^{e_m} = H_{K[\bar{x}]^{\Gamma}}(t) \cdot \prod_{i=1}^N (1-t^{\deg(p_i)}).
    \end{equation*}

    Combined with Klein's result in equation (\ref{eq_Klein}) below, this is enough to compute maps whose automorphism groups contain $\Gamma$ in dimension $1$. In higher dimensions, we need invariant $N$-forms and to apply Theorem \ref{thm_form_aut}. We now discuss the needed theory \cite[Ch. 9]{smith} or \cite[Ch. 5]{Benson}.

    The action of $\Gamma$ on the vector space $V$ extends not just to $K[V]$, as we have seen, but also to the exterior algebra $E[V]$. Since we have an action on both $K[V]$ and $E[V]$, then we have an action on the tensor $K[V] \otimes E[V]$; see, for example, Soloman \cite{Soloman}. We think of $K[V] \otimes E[V]$ as the \emph{algebra of polynomial differential forms} on $V$ by introducing the exterior derivative $d$.  With $K[V] = K[x_1,\ldots,x_N]$, then $K[V] \otimes E[V]$ can be thought of as the Grassmann algebra of differential forms $K[x_1,\ldots,x_N] \otimes \Lambda(dx_1,\ldots, dx_N)$ so that a typical element is a linear combination of terms of the form $fdx_{i_1} \wedge \cdots \wedge dx_{i_j}$, \cite[Ch. 5]{Benson}.  Thus, elements of $(K[V] \otimes E[V])^{\Gamma}$ are $\Gamma$ invariant differential forms on $V$.  Furthermore, since $d$ commutes with the action of $\Gamma$, we have if $f \in K[V]^{\Gamma}$, then $df \in (K[V] \otimes E[V])^{\Gamma}$. In particular, if $F$ is an invariant $p$-form of degree $m$, then $dF$ is an invariant $(p+1)$-form of degree $m-1$. Given also an invariant $q$-form $G$ of degree $n$, then $F \wedge G$ is an invariant $(p+q)$-form of degree $m+n$.

    There is a 2-variable Molien series for $N$-forms.
    \begin{defn}
        The \emph{Molien series} for $s$-forms of degree $t$ for a finite group $\Gamma$ and a linear group character $\chi$ is given by
        \begin{equation*}
            \frac{1}{\abs{\Gamma}}\sum_{\gamma \in \Gamma} \frac{\chi(\gamma)\det(1+s\gamma)}{\det(1-t\gamma)},
        \end{equation*}
        where $s$ is the rank of the form and $t$ is the polynomial degree.
        This was shown by Solomon \cite[p. 270]{smith2} to be the Hilbert-Poincar\'e Series of $(K[V] \otimes E[V])^{\Gamma}$ generated by $\chi$-invariants.  That is,
        \begin{equation*}
            \frac{1}{\abs{\Gamma}}\sum_{\gamma \in \Gamma} \frac{\chi(\gamma)\det(1+s\gamma)}{\det(1-t\gamma)} = \sum_{p=0}^N \left( \sum_{m=0}^{\infty} \left(\dim \bigwedge_p(K^N)^{\Gamma}_m\right)t^m\right)s^p.
        \end{equation*}
    \end{defn}
    It is also possible to define a Reynolds operator. Thus, we can determine the degrees of the invariant $N$-forms with the Molien series, find them with the Reynolds operator, and convert them to endomorphisms of $\P^N$ with Theorem \ref{thm_form_aut}.

    Note that since there are $(N+1)$ primary invariants $p_0,\ldots,p_N$ for a finite subgroup of $\PGL_{N+1}$. Then we can always construct an invariant $(N+1)$-form $dp_0 \wedge \cdots \wedge dp_N$.

 \subsection{Equivariants}
    We can also attack the problem of maps with nontrivial automorphisms directly with invariant theory by studying \emph{equivariants} \cite{gatermann, Worfolk}. We follow the presentation of Worfolk \cite{Worfolk}. Let $\Gamma$ be a finite group and $V,W$ be finite dimensional vector spaces with $\rho_V:\Gamma \to GL(V)$ and $\rho_W: \Gamma \to \GL(W)$. We have already seen the action on $K[V]$ and $K[W]$ as
    \begin{equation*}
        (\gamma f)(v) = f(\rho(\gamma^{-1})v).
    \end{equation*}
    We can identify the space of polynomial mappings from $V$ to $W$ with $K[V] \otimes W$.
    \begin{defn}
        The polynomial mappings that commute with $\Gamma$ are called \emph{equivariants} (or sometimes \emph{covariants}), and we denote them as
        \begin{equation*}
            (K[V] \otimes W)^{\Gamma} = \{ g \in K[V] \otimes W \col g \circ \rho_V(\gamma) = \rho_W(\gamma)g\}.
        \end{equation*}
    \end{defn}
    \begin{rem}
        Note that $g$ being an equivariant is the same as $g$ being invariant under the action of conjugation.
    \end{rem}
    As with invariants, there is a finite set of generators for equivariants as a $K[V]^{\Gamma}$ module and there is a similar Neother bound of $\abs{\Gamma}-1$ for the degree of the generators \cite[Proposition 4.3]{Worfolk}.
    Thus, the equivariant mappings form a module over the ring of invariants and, for finite groups, the module of equivariants is finitely generated by a set of (computable) fundamental equivariants,        i.e., $(K[V] \otimes W)^{\Gamma}$ is a Cohen-Macaulay module.

    \begin{prop}{\cite[Proposition 4.8]{Worfolk}} \label{prop_fund_equi}
        For $\Gamma$ finite and $p_1,\ldots,p_N$ a set of primary invariants for $K[V]^{\Gamma}$ with degrees $d_1,\ldots, d_N$, the number of fundamental equivariants $s$ is given by
        \begin{equation*}
            s = m\frac{d_1\cdots d_N}{\abs{\Gamma}}
        \end{equation*}
        where $m = \dim(W)$. The degrees $e_1,\ldots, e_s$ (with multiplicity) of the fundamental equivariants are given by \begin{equation*}
            t^{e_1} + \cdots + t^{e_s} = \Psi(t)\cdot(1-t^{d_1}) \cdots (1-t^{d_N}),
        \end{equation*}
        where $\Psi(t)$ is the equivariant Molien series.
    \end{prop}

    We can find the finitely many generators with a Reynolds operator calculation. The existence of an equivariant Molien series makes this process somewhat easier.

    \begin{prop}
        The \emph{equivariant Molien series} for a finite group $\Gamma$ and a linear group character $\chi$ is given by
        \begin{equation*}
            \Psi(t) = \frac{1}{\abs{\Gamma}}\sum_{\gamma \in \Gamma} \frac{\chi(\gamma)\tr(\gamma^{-1})}{\det(1-t\gamma)}
        \end{equation*}
        and is the Hilbert-Poincar\'e series for $\chi$-equivariants \cite{gatermann} or \cite[Theorem 4.6]{Worfolk}.
    \end{prop}
    \begin{defn}
        The \emph{equivariant Reynolds operator} projects onto equivariants for a linear finite group $\Gamma$ and a group character $\chi$ and is given by
        \begin{equation*}
            RE_{\chi}(f) = \frac{1}{\abs{\Gamma}}\sum_{\gamma \in \Gamma} \chi(\gamma) \gamma^{-1}  f  \gamma.
        \end{equation*}
    \end{defn}

\section{Connections to Automorphisms} \label{sect.inv_and_auto}
    We turn now to the question of dynamical systems with nontrivial stabilizer groups.  We notate the tuple $(x_0,\ldots,x_N)$ as $\overline{x}$.

    Given a morphism
    $f \colon \P^N_K \to \P^N_K$ of degree~$d$, we can represent $f$ as an $N+1$-tuple of homogeneous polynomials of degree $d$, with no common roots. Then we may define a \emph{lift} of $f$ by
    \begin{align*}
        \Phi \colon \A^{N+1} &\to \A^{N+1}\\
        (x_0, \ldots , x_N) &\mapsto (f_0(\overline x), \ldots, f_N(\overline x) ).
    \end{align*}

    Let $\Gamma$ be a finite subgroup of $\PGL_{N+1}$ and let $\overline \Gamma$ be its preimage in $\SL_{N+1}$.
    Thinking of $\Phi$ as a vector field on $\A^{N+1}$, we see that an element $\gamma \in \overline \Gamma$ acts on an element $\Phi$ by conjugation:
   \[
   \Phi^{\gamma}(\bar{x}) = \gamma \cdot \left(\Phi(\gamma^{-1}\cdot \bar{x}^T) \right)^T.
   \]
    So we may consider the action of elements of $\SL_{N+1}$ (or more generally $\GL_{N+1}$) on homogeneous polynomials in $N+1$ variables.

\subsection{Automorphisms on $\P^1$} \label{sect.4.1}

    For a map $f:\P^1 \to \P^1$, it was known as early as Klein \cite[p. 345]{klein} that given an invariant $G$ of $\Gamma$, the map
    \begin{equation} \label{eq_Klein}
        f_G = \left[-\frac{\partial G}{ \partial y}, \frac{\partial G}{ \partial x}\right]
    \end{equation}
    has $\Gamma \subseteq \Aut(f)$. Note that since we are working over projective space, $G$ can be a relative invariant.  Also, note that if $G$ has only simple zeros, then $f_G$ is a morphism of degree $(\deg(G)-1)$ \cite{HutzSzpiro}.

    Doyle-McMullen gave the general correspondence between polynomial invariants and invariant homogeneous $1$-forms. $\Gamma$-invariant homogeneous 1-forms $f_1dx + f_2dy$ correspond to maps $f = [-f_2,f_1]$ with $\Gamma \subseteq \Aut(f)$ \cite[Proposition 5.1]{Doyle}. The general connection between $n$-forms and automorphisms will be proven in Theorem \ref{thm_form_aut}.
    \begin{prop}{\cite[Theorem 5.2]{Doyle}} \label{prop_DM}
        A homogeneous 1-form $\theta$ is invariant if and only if
        \begin{equation*}
            \theta = F \lambda + dG,
        \end{equation*}
        where $F$ and $G$ are invariant homogeneous polynomials with the same character, $\deg{G} = \deg{F}+2$, and $\lambda = (xdy - ydx)/2$.
    \end{prop}
    \begin{rem}
        Proposition \ref{prop_DM} is stated as by Doyle-McMullen \cite{Doyle}, but it needs a slight modification. The degree restriction to ensure homogeneity omits the case of Klein, which is $F=0$. The case $G=0$ is also possible and gives the identity map on projective space.
    \end{rem}
    For invariants $F$ and $G$, we have the corresponding map
    \begin{equation*}
        f = [xF/2 + G_y, yF/2 - G_x].
    \end{equation*}

    Thus, given any finite subgroup $\Gamma \subset \PGL_2$, we have 3 options for determining maps with $\Gamma \subset \Aut(f)$.
    \begin{enumerate}
        \item Compute $K[\bar{x}]^{\Gamma}$ and use Doyle-McMullen
        \item Use the equivariant Molien Series and the equivariant Reynolds operator
        \item Compute generators for the module of equivariants
    \end{enumerate}
    See Section \ref{sect_exmp_dim1} for examples on $\P^1$.

    Silverman modernized the classification of finite subgroups of $\PGL_2$ \cite{Silverman12}. So we know that any automorphism group must be conjugate to one of the following, where $\zeta_n$ represents a primitive $n\tth$ root of unity.
    \begin{enumerate}
        \item Cyclic group of order $n$:
        \begin{equation*}
            C_n = \left\langle \begin{pmatrix}
              \zeta_n & 0 \\ 0 & \zeta_n^{-1}
            \end{pmatrix} \right\rangle.
        \end{equation*}
\begin{code}
\begin{verbatim}
K<a> := CyclotomicField(3);
H := MatrixGroup< 2, K| [a,0, 0,1/a] >;
\end{verbatim}
\end{code}

        \item Dihedral Group of order $2n$:
            \begin{equation*}
                D_{2n} = \left\langle \begin{pmatrix}
                  \zeta_n & 0 \\ 0 & \zeta_n^{-1}
                \end{pmatrix},
                 \begin{pmatrix}
                  0 & 1 \\ 1 &0
                \end{pmatrix}
                \right\rangle.
            \end{equation*}
\begin{code}
\begin{verbatim}
K<a> := CyclotomicField(3);
R<x> := PolynomialRing(K);
f := x^2+1;
L<i> := ext<K|f>;
H := MatrixGroup< 2, L| [a,0, 0,1/a], [0,i, i,0] >;
\end{verbatim}
\end{code}

        \item Tetrahedral Group of order $12$:
            \begin{equation*}
                A_4 =
                \left\langle
                 \begin{pmatrix}
                  i & i \\ 1& -1
                \end{pmatrix}, \begin{pmatrix}
                  -1 & 0 \\ 0& 1
                \end{pmatrix},
 		         \begin{pmatrix}
                  0 & 1 \\ 1 &0
                \end{pmatrix} \right\rangle.
            \end{equation*}

\begin{code}
\begin{verbatim}
R<x> := PolynomialRing(Integers());
f := x^2+1;
K<i> := NumberField(f);
H := MatrixGroup< 2, K| [(-1+i)/2,(-1+i)/2, (1+i)/2,(-1-i)/2], [0,i, -i,0]  >;
\end{verbatim}
\end{code}
        \item Octahedral Group of order $24$:
            \begin{equation*}
                S_4 =
                \left\langle
                 \begin{pmatrix}
                  i & i \\ 1& -1
                \end{pmatrix}, \begin{pmatrix}
                  i & 0 \\ 0& 1
                \end{pmatrix},
 		        \begin{pmatrix}
                  0 & 1 \\ 1 &0
                \end{pmatrix} \right\rangle.
            \end{equation*}

\begin{code}
\begin{verbatim}
R<x> := PolynomialRing(Integers());
f := x^2+1;
K<i> := NumberField(f);
R<x> := PolynomialRing(K);
f := x^2-2;
L<a> := ext<K|f>;

H := MatrixGroup< 2, L| [(-1+i)/2,(-1+i)/2, (1+i)/2,(-1-i)/2], [(1+i)/a,0, 0,(1-i)/a]  >;


R<x> := PolynomialRing(Integers());
f := x^2+1;
K<i> := NumberField(f);
R<x> := PolynomialRing(K);
f := x^2+2;
L<a> := ext<K|f>;
H2 := MatrixGroup< 2, L| [i/a,i/a, 1/a,-1/a], [i,0, 0,1], [0,1, 1,0]  >;
\end{verbatim}
\end{code}

        \item Icosahedral Group of order $60$:
            \begin{equation*}
                A_5 =
                \left\langle
                 \begin{pmatrix}
                  \zeta_5+\zeta_5^{-1} & 1 \\ 1& -\zeta_5-\zeta_5^{-1}
                \end{pmatrix}, \begin{pmatrix}
                  \zeta_5 & 0 \\ 0& 1
                \end{pmatrix},
 		\begin{pmatrix}
                  0 & 1 \\ -1 &0
                \end{pmatrix}
 \right\rangle.
            \end{equation*}

\begin{code}
\begin{verbatim}
K<z5>:=CyclotomicField(5);
R<x> := PolynomialRing(K);
f:=x^2+1;
L<i>:=ext<K | f>;
R<x,y> := PolynomialRing(L,2);
X:=Matrix(R,2,1,[x,y]);
a:=2*z5^3 + 2*z5^2 + 1; //sqrt(5)

H := MatrixGroup< 2, L| [z5^3,0, 0,z5^2], [0,1, -1,0], [(z5^4-z5)/a, (z5^2-z5^3)/a, (z5^2-z5^3)/a, -(z5^4-z5)/a]  >;

t:=Matrix(R,2,[z5 + z5^4, 1, 1, -(z5+z5^4)]);
//b is sqrt of det of t
b := -z5^3 + z5^2;
//silverman
H := MatrixGroup< 2, L| [z5,0, 0,1], [0,1, -1,0], [(z5 + z5^4)/b, 1/b, 1/b, (-(z5+z5^4))/b]  >;
\end{verbatim}
\end{code}
    \end{enumerate}

    Silverman \cite{Silverman12} gave explicit descriptions for the infinite family of maps with a stabilizer group containing $C_n$ and $D_{2n}$.

\subsection{Examples on $\P^1$} \label{sect_exmp_dim1}

    \begin{exmp}[Octahedral Group] \label{exmp.octahedral}
        The subgroup of $\PGL_{2}$ of order $24$ has a representation in $\SL_2$:
        \begin{equation*}
            \Gamma = \left\langle \frac{1}{2}\begin{pmatrix}
              -1+i & -1+i \\ 1+i & -1-i
            \end{pmatrix}, \frac{1}{\sqrt{2}}\begin{pmatrix}
              1+i & 0 \\ 0& 1-i
            \end{pmatrix} \right\rangle.
        \end{equation*}
        We first compute some invariants with the Molien series. Using the trivial character, we obtain the series
        \begin{equation*}
            1 + t^8 + t^{12} + t^{16} + t^{18} + O(t^{20}).
        \end{equation*}
        With a search of degrees $8$ and $12$ with the Reynolds operator, we can find the two invariants
        \begin{align*}
            F&=x^8 + 14x^4y^4 + y^8\\
            G&=x^{10}y^2 - 2x^6y^6 + x^2y^{10} = (x^5y - xy^5)^2.
        \end{align*}
        With the other linear character we produce a (relative) Molien series
        \begin{equation*}
            t^6 + t^{12} + t^{14} + t^{18} + t^{20} + t^{22} + t^{24} + O(t^{26}).
        \end{equation*}
        We can again find the corresponding invariants:
        \begin{align*}
            R_1&=x^5y - xy^5  \\
            R_2&=x^{12} - 33x^8y^4 - 33x^4y^8 + y^{12}\\
            R_3&=x^{13}y + 13x^9y^5 - 13x^5y^9 - xy^{13} = FR_1.
        \end{align*}
        Using the observation of Klein, we can determine the maps with $T \subseteq \mathcal{A}_{\phi}$. Notice that it is possible to have different invariants result in the same morphism because we are working with projective morphisms and can have cancellation of common factors in the coordinates. This cancellation is how the degree $12$ invariant $G$ corresponds to a degree $5$ map.
        \begin{equation*}
        \begin{tabular}{|l|l|}
            \hline
            Generating Invariant & Morphism\\
            \hline
            $(R_1)$ or $(G)$ & $[-x^5 + 5xy^4,5x^4y - y^5]$\\
            \hline
            $(F)$ & $[-7x^4y^3 - y^7,x^7 + 7x^3y^4]$\\
            \hline
            $(F\cdot(xy))$ & $[x^9 + 70x^5y^4 + 9xy^8,-9x^8y - 70x^4y^5 - y^9]$\\
            \hline
            $(R_2)$ & $[11x^8y^3 + 22x^4y^7 - y^{11},x^{11} - 22x^7y^4 - 11x^3y^8]$\\
            \hline
            $(R_3)$ & $[-x^{13} - 65x^9y^4 + 117x^5y^8 + 13xy^{12}, 13x^{12}y + 117x^8y^5 - 65x^4y^9 - y^{13}]$\\
            \hline
            $(F^2 + R_3)\cdot (xy)$ & $[x^{14}y - 56x^{12}y^3 + 39x^{10}y^5 - 792x^8y^7 - 65x^6y^9 - 168x^4y^{11} - 7x^2y^{13} - 8y^{15}$,\\
            & $8x^{15} - 7x^{13}y^2 + 168x^{11}y^4 - 65x^9y^6 + 792x^7y^8 + 39x^5y^{10} + 56x^3y^{12} + xy^{14}]$\\
            \hline
            $(R_1R_2)$ & $[-x^{17} + 170x^{13}y^4 - 442x^5y^{12} + 17xy^{16}, 17x^{16}y - 442x^{12}y^5 + 170x^4y^{13} - y^{17}]$\\
            \hline
        \end{tabular}
        \end{equation*}

        Note that we could also have proceeded with the equivariant Molien series.
 	    With the trivial character, we have the series
        \begin{equation*}
            t + t^7 + t^9 + t^{11} + t^{13} + t^{15} + 2t^{17} + 2t^{19} + O(t^{21}).
        \end{equation*}
        With the other linear character we have the (relative) series
        \begin{equation*}
            t^5 + t^7 + t^{11} + 2t^{13} + t^{15} + t^{17} + 2t^{19} + 2t^{21} + 2t^{23} + O(t^{25}).
        \end{equation*}
        Notice that the exponents in the equivariant series do, in fact, correspond to the degrees of the maps with $\Gamma \subseteq \Aut(f)$.

\begin{code}
{\footnotesize \begin{verbatim}
//Magma code:

L<i> := CyclotomicField(8);
R<x> := PolynomialRing(L);
a:=i^3-i;

H := MatrixGroup< 2, L| [(-1+i^2)/2,(-1+i^2)/2, (1+i^2)/2,(-1-i^2)/2], [(1+i^2)/a,0, 0,(1-i^2)/a]  >;

//Compute the Molien series for linear characters.
//Only CT[1],CT[2] are linear.

CT:=CharacterTable(H);
xi:=CT[1];
S:=Inverse(NumberingMap(H));
R<t>:=PolynomialRing(L);
Mol:=0;
for i:=1 to Order(H) do
  Mol := Mol + (L!xi(S(i)))*Determinant(S(i))/(Determinant(DiagonalMatrix([1,1])- t*Matrix(R,S(i))));
end for;
Mol:=Mol/Order(H);
Mol;
PS<t>:=PowerSeriesRing(L);
PS!Mol;


//Apply the Reynold operator to find a particular invariant
R<x,y>:=PolynomialRing(L,2);
X:=Matrix(R,2,1,[x,y]);
M:=MonomialsOfDegree(R,8);
INV:={};

for i:=1 to #M do
  T:=0;
  F:=M[i];
  for i:= 1 to #H do
    t:=Matrix(R,S(i))*X;
    G:=(L!xi(S(i)))*Evaluate(F,[t[1,1],t[2,1]]);
    T:=T+G;
  end for;
  INV:=Include(INV, T);
end for;
INV;
//
//-----------------------------------------------------------------------------
//Compute the equivariant Molien Series
xi:=CT[1];
S:=Inverse(NumberingMap(H));
R<t>:=PolynomialRing(L);
Mol:=0;
for i:=1 to Order(H) do
  Mol := Mol + (L!xi(S(i)))*Trace(Matrix(R,S(i)^(-1)))/(Determinant(DiagonalMatrix([1,1])- t*Matrix(R,S(i))));
end for;
Mol:=Mol/Order(H);
Mol;
PS<t>:=PowerSeriesRing(L);
PS!Mol;

//Apply the equivariant Reynolds operator
R<x,y> := PolynomialRing(L,2);
X:=Matrix(R,2,1,[x,y]);
M:=MonomialsOfDegree(R,7);
Auts:={};


CT:=CharacterTable(H);
xi:=CT[1];

for xxi:=1 to #M do
  for yi:=1 to #M do
      T:=Matrix(2,1,[0,0]);
      F:=Matrix(R,2,1,[M[xxi],M[yi]]);
      for i:= 1 to #H do
        t:=Matrix(R,S(i))*X;
        G:=Matrix(R,2,1,[Evaluate(F[1,1],[t[1,1],t[2,1]]),Evaluate(F[2,1],[t[1,1],t[2,1]])]);
        C:=Matrix(R,S(i))^(-1)*G;
        T:=T+R!L!xi(S(i))*C;
      end for;
      g:=GCD([T[1,1],T[2,1]]);
     if g ne 0 then
        T[1,1]:=T[1,1]/g;
        T[2,1]:=T[2,1]/g;
      end if;
      Auts:=Include(Auts, T);
    end for;
  end for;
Auts;
\end{verbatim}}
\end{code}
    \end{exmp}

\begin{code}
\begin{verbatim}
Tetrahedral Group
The resulting fundamental invariants are the same as is Miller, Blichfeldt, Dickson.

L<w> := CyclotomicField(12);
i:=w^3;
H := MatrixGroup< 2, L| [(-1+i)/2,(-1+i)/2, (1+i)/2,(-1-i)/2], [0,i, -i,0]  >;
CT:=CharacterTable(H);  //first 6 have no zeros

xi:=CT[5];
S:=Inverse(NumberingMap(H));
R<t>:=PolynomialRing(L);
Mol:=0;
for i:=1 to Order(H) do
  Mol := Mol + (R!xi(S(i)))/(Determinant(DiagonalMatrix([1,1])- t*Matrix(R,S(i))));
end for;
Mol:=Mol/Order(H);
Mol;
PS<t>:=PowerSeriesRing(L);
PS!Mol;

CT[1];
1 + t^8 + 2*t^12 + t^16 + O(t^20)

x^8 + 14*x^4*y^4 + y^8
-x^12 + 33*x^8*y^4 + 33*x^4*y^8 - y^12
x^10*y^2 - 2*x^6*y^6 + x^2*y^10

CT[2]
t^6 + t^14 + 2*t^18 + t^22 + O(t^26)
**x^5*y-x*y^5
x^13*y + 13*x^9*y^5 - 13*x^5*y^9 - x*y^13

CT[3]
t^10 + t^14 + t^18 + 2*t^22 + 2*t^26 + O(t^30)
x^9*y + (-4*w^2 + 2)*x^7*y^3 + (4*w^2 - 2)*x^3*y^7 - x*y^9

CT[4]
t^4 + t^8 + t^12 + 2*t^16 + 2*t^20 + O(t^24)
**x^4 + (4*w^2 - 2)*x^2*y^2 + y^4
-x^8 + (8*w^2 - 4)*x^6*y^2 + 10*x^4*y^4 + (8*w^2 - 4)*x^2*y^6 - y^8

CT[5]
t^4 + t^8 + t^12 + 2*t^16 + 2*t^20 + O(t^24)
**x^4 + (-4*w^2 + 2)*x^2*y^2 + y^4,
-x^8 + (-8*w^2 + 4)*x^6*y^2 + 10*x^4*y^4 + (-8*w^2 + 4)*x^2*y^6 - y^8

CT[6]
t^10 + t^14 + t^18 + 2*t^22 + 2*t^26 + O(t^30)
x^9*y + (4*w^2 - 2)*x^7*y^3 + (-4*w^2 + 2)*x^3*y^7 - x*y^9

R<x,y>:=PolynomialRing(L,2);
X:=Matrix(R,2,1,[x,y]);
M:=MonomialsOfDegree(R,8);
INV:={};

for j:=1 to #M do
  T:=0;
  F:=M[j];
  for k:= 1 to #H do
    t:=Matrix(R,S(k))*X;
    G:=(R!xi(S(k)))*Evaluate(F,[t[1,1],t[2,1]]);
    T:=T+G;
  end for;
  INV:=Include(INV, T);
end for;
INV;


R<x,y> := PolynomialRing(L,2);
X:=Matrix(R,2,1,[x,y]);
f1:=x^8 + 14*x^4*y^4 + y^8;
f2:=-x^12 + 33*x^8*y^4 + 33*x^4*y^8 - y^12;
f3:=x^10*y^2 - 2*x^6*y^6 + x^2*y^10;

g1:=x^5*y-x*y^5;  //***
g2:=x^13*y + 13*x^9*y^5 - 13*x^5*y^9 - x*y^13;

h1:=x^9*y + (-4*w^2 + 2)*x^7*y^3 + (4*w^2 - 2)*x^3*y^7 - x*y^9;

j1:=x^4 + (4*w^2 - 2)*x^2*y^2 + y^4;//***
j2:=-x^8 + (8*w^2 - 4)*x^6*y^2 + 10*x^4*y^4 + (8*w^2 - 4)*x^2*y^6 - y^8;

k1:=x^4 + (-4*w^2 + 2)*x^2*y^2 + y^4;//***
k2:=-x^8 + (-8*w^2 + 4)*x^6*y^2 + 10*x^4*y^4 + (-8*w^2 + 4)*x^2*y^6 - y^8;

l1:=x^9*y + (4*w^2 - 2)*x^7*y^3 + (-4*w^2 + 2)*x^3*y^7 - x*y^9;

f:=j1^2*j2;

F:=Matrix(R,2,1,[-Derivative(f,2),Derivative(f,1)]);
g:=GCD([F[1,1],F[2,1]]);
g;
F[1,1]:=F[1,1]/g;
F[2,1]:=F[2,1]/g;
F;
t:=Matrix(R,H.1)*X;
G:=Matrix(R,2,1,[Evaluate(F[1,1],[t[1,1],t[2,1]]),Evaluate(F[2,1],[t[1,1],t[2,1]])]);
C:=Matrix(R,H.1)^(-1)*G;
C;
t:=Matrix(R,H.2)*X;
G:=Matrix(R,2,1,[Evaluate(F[1,1],[t[1,1],t[2,1]]),Evaluate(F[2,1],[t[1,1],t[2,1]])]);
C:=Matrix(R,H.2)^(-1)*G;
C;

f1
[-7*x^4*y^3 - y^7]
[ x^7 + 7*x^3*y^4]

f2
[11*x^8*y^3 + 22*x^4*y^7 - y^11]
[x^11 - 22*x^7*y^4 - 11*x^3*y^8]

f3
[ x^5 - 5*x*y^4]
[-5*x^4*y + y^5]

g1
[-x^5 + 5*x*y^4]
[ 5*x^4*y - y^5]

g2
[-x^13 - 65*x^9*y^4 + 117*x^5*y^8 + 13*x*y^12]
[ 13*x^12*y + 117*x^8*y^5 - 65*x^4*y^9 - y^13]

h1
[-x^9 + (12*w^2 - 6)*x^7*y^2 + (-28*w^2 + 14)*x^3*y^6 + 9*x*y^8]
[ 9*x^8*y + (-28*w^2 + 14)*x^6*y^3 + (12*w^2 - 6)*x^2*y^7 - y^9]

j1
[(-2*w^2 + 1)*x^2*y - y^3]
[ x^3 + (2*w^2 - 1)*x*y^2]

j2
[(-2*w^2 + 1)*x^2*y + y^3]
[-x^3 + (2*w^2 - 1)*x*y^2]

k1
[ (2*w^2 - 1)*x^2*y - y^3]
[x^3 + (-2*w^2 + 1)*x*y^2]

k2
[  (2*w^2 - 1)*x^2*y + y^3]
[-x^3 + (-2*w^2 + 1)*x*y^2]

l1
[-x^9 + (-12*w^2 + 6)*x^7*y^2 + (28*w^2 - 14)*x^3*y^6 + 9*x*y^8]
[ 9*x^8*y + (28*w^2 - 14)*x^6*y^3 + (-12*w^2 + 6)*x^2*y^7 - y^9]


f1=j1*k1
g2=g1*f1
g2=k1*l1
\end{verbatim}
\end{code}

    \begin{exmp}\label{exmp.equi.oct}
        We again start with the octahedral group but now use the module of equivariants to determine maps.

        We compute the equivariant Molien Series
        \begin{equation*}
            \frac{(t^9 - t^7 + t^5 - t^3 + t)}{(t^{12} - t^{10} + t^8 - 2t^6 + t^4 - t^2 + 1)} = t + t^7 + t^9 + t^{11} + t^{13} + t^{15} + 2t^{17} + 2t^{19} + O(t^{21}).
        \end{equation*}
        We see there are four fundamental equivariants and get their degrees as $1$, $7$, $11$, and $17$:
        \begin{equation*}
            \frac{(t^9 - t^7 + t^5 - t^3 + t)}{(t^{12} - t^{10} + t^8 - 2t^6 + t^4 - t^2 + 1)} \cdot (1-t^8)(1-t^{12}) = t + t^7 + t^{11} + t^{17}.
        \end{equation*}
        The degree $7$ and $5$ equivariants come directly from (\ref{eq_Klein}) and the primary invariants
        \begin{align*}
            p_8=&x^8 + 14x^4y^4 + y^8\\
            p_{12}=&x^{10}y^2 - 2x^6y^6 + x^2y^{10} = (x^5y - xy^5)^2,
        \end{align*}
        where the repeated factor in the degree $12$ invariant causes the resulting map in lowest form to be degree $5$, not $11$. We use the Reynold's operator to find the degree $17$ equivariant. The four fundamental equivariants are
        \begin{align*}
            f_1(x,y)=&(x:y)\\
            f_7(x,y)=&(7x^4y^3 + y^7 : -x^7 - 7x^3y^4)\\
            f_5(x,y)=&(-x^5 + 5xy^4 : 5x^4y - y^5)\\
            f_{17}(x,y)=&(x^{17} - 60x^{13}y^4 + 110x^9y^8 + 212x^5y^{12} - 7xy^{16}\\ &:-7x^{16}y + 212x^{12}y^5 + 110x^8y^9 - 60x^4y^{13} + y^{17}).
        \end{align*}
        Now we can find a new element in the module by combining equivariants and invariants. We can use either absolute or relative invariants (of the same character) for projective maps since the common factors among the coordinates can be removed. In fact, the $\chi$-relative invariants are a module over $K[p_1,\ldots,p_N]$ \cite{Kemper}. Now we can find a new element in the module as
        \begin{align*}
            f_{17} + 2p_{12}f_5 &=\\
            &(x^{17} + 2x^{15}y^2 - 60x^{13}y^4 - 14x^{11}y^6 + 110x^9y^8 +22x^7y^{10} + 212x^5y^{12} - 10x^3y^{14} - 7xy^{16}\\
            &: -7x^{16}y - 10x^{14}y^3 + 212x^{12}y^5 + 22x^{10}y^7 + 110x^8y^9 - 14x^6y^{11} -60x^4y^{13} + 2x^2y^{15} + y^{17})
        \end{align*}

\begin{code}
{\footnotesize
\begin{verbatim}
//Magma code:
L<i> := CyclotomicField(8);
R<x> := PolynomialRing(L);
a:=i^3-i;

H := MatrixGroup< 2, L| [(-1+i^2)/2,(-1+i^2)/2, (1+i^2)/2,(-1-i^2)/2], [(1+i^2)/a,0, 0,(1-i^2)/a]  >;

//Compute the Molien series for linear characters.
//Only CT[1],CT[2] are linear.

CT:=CharacterTable(H);

//Compute the equivariant Molien Series
xi:=CT[1];
S:=Inverse(NumberingMap(H));
R<t>:=PolynomialRing(L);
EqMol:=0;
for i:=1 to Order(H) do
  EqMol := EqMol + (L!xi(S(i)))*Trace(Matrix(R,S(i)^(-1)))/(Determinant(DiagonalMatrix([1,1])- t*Matrix(R,S(i))));
end for;
EqMol:=EqMol/Order(H);
EqMol;
(EqMol)*((1-t^8)*(1-t^12));



//Apply the equivariant Reynolds operator
R<x,y> := PolynomialRing(L,2);
X:=Matrix(R,2,1,[x,y]);
M:=MonomialsOfDegree(R,11);
Auts:={};


xi:=CT[1];
S:=Inverse(NumberingMap(H));
for xxi:=1 to #M do
  for yi:=1 to #M do
      T:=Matrix(2,1,[0,0]);
      F:=Matrix(R,2,1,[M[xxi],M[yi]]);
      for i:= 1 to #H do
        t:=Matrix(R,S(i))*X;
        G:=Matrix(R,2,1,[Evaluate(F[1,1],[t[1,1],t[2,1]]),Evaluate(F[2,1],[t[1,1],t[2,1]])]);
        C:=Matrix(R,S(i))^(-1)*G;
        T:=T+R!L!xi(S(i))*C;
      end for;
      g:=GCD([T[1,1],T[2,1]]);
      print g;
      if g ne 0 then
        T[1,1]:=T[1,1]/g;
        T[2,1]:=T[2,1]/g;
      end if;
      Auts:=Include(Auts, T);
    end for;
  end for;
Auts;

\end{verbatim}
}
\end{code}

\begin{code}
{\footnotesize
\begin{verbatim}
//Magma code:
K<a> := CyclotomicField(7);
R<x> := PolynomialRing(K);
f := x^2+1;
L<i> := ext<K|f>;
H := MatrixGroup< 2, L| [a,0, 0,1/a], [0,1, 1,0] >;

PrimaryInvariants(InvariantRing(H));

//Compute the Molien series for linear characters.
//Only CT[1],CT[2] are linear.

CT:=CharacterTable(H);

//Compute the equivariant Molien Series
xi:=CT[1];
S:=Inverse(NumberingMap(H));
R<t>:=PolynomialRing(L);
EqMol:=0;
for i:=1 to Order(H) do
  EqMol := EqMol + (L!xi(S(i)))*Trace(Matrix(R,S(i)^(-1)))/(Determinant(DiagonalMatrix([1,1])- t*Matrix(R,S(i))));
end for;
EqMol:=EqMol/Order(H);
EqMol;
(EqMol)*((1-t^2)*(1-t^7));



//Apply the equivariant Reynolds operator
R<x,y> := PolynomialRing(L,2);
X:=Matrix(R,2,1,[x,y]);
M:=MonomialsOfDegree(R,11);
Auts:={};


xi:=CT[1];
S:=Inverse(NumberingMap(H));
for xxi:=1 to #M do
  for yi:=1 to #M do
      T:=Matrix(2,1,[0,0]);
      F:=Matrix(R,2,1,[M[xxi],M[yi]]);
      for i:= 1 to #H do
        t:=Matrix(R,S(i))*X;
        G:=Matrix(R,2,1,[Evaluate(F[1,1],[t[1,1],t[2,1]]),Evaluate(F[2,1],[t[1,1],t[2,1]])]);
        C:=Matrix(R,S(i))^(-1)*G;
        T:=T+R!L!xi(S(i))*C;
      end for;
      g:=GCD([T[1,1],T[2,1]]);
      print g;
      if g ne 0 then
        T[1,1]:=T[1,1]/g;
        T[2,1]:=T[2,1]/g;
      end if;
      Auts:=Include(Auts, T);
    end for;
  end for;
Auts;

\end{verbatim}
\end{code}

\begin{code}
\begin{verbatim}
//Magma code:

L<i> := CyclotomicField(8);
R<x> := PolynomialRing(L);
a:=i^3-i;

H := MatrixGroup< 2, L| [(-1+i^2)/2,(-1+i^2)/2, (1+i^2)/2,(-1-i^2)/2], [(1+i^2)/a,0, 0,(1-i^2)/a]  >;
S:=Inverse(NumberingMap(H));
CT:=CharacterTable(H);

//MolienSeries
xi:=CT[1];
R<t>:=PolynomialRing(L);
Mol:=0;
for i:=1 to Order(H) do
  Mol := Mol + (L!xi(S(i)))/(Determinant(DiagonalMatrix([1,1])- t*Matrix(R,S(i))));
end for;
Mol:=Mol/Order(H);
Mol;
PS<t>:=PowerSeriesRing(L);
PS!Mol;

//Apply the Reynold operator to find a particular invariant
R<x,y>:=PolynomialRing(L,2);
X:=Matrix(R,2,1,[x,y]);
M:=MonomialsOfDegree(R,12);
INV:={};

CT:=CharacterTable(H);
xi := CT[2];

for i:=1 to #M do
  T:=0;
  F:=M[i];
  for i:= 1 to #H do
    t:=Matrix(R,S(i))*X;
    G:=(L!xi(S(i)))*Evaluate(F,[t[1,1],t[2,1]]);
    T:=T+G;
  end for;
  INV:=Include(INV, T);
end for;
INV;
\end{verbatim}
\end{code}
    \end{exmp}

\subsection{Higher Dimensions}
    From Doyle-McMullen \cite{Doyle} there is a correspondence between invariants and equivariants for endomorphism of $\P^1$. From Crass \cite{Crass} we have a 1-1 correspondence between invariant 2-forms and maps with automorphisms for unitary representations. Note that Maschke's Theorem says that we can always reduce to considering unitary representations.
    We prove the more general connection between $n$-forms and equivariants. This was stated but not proven by Crass \cite{Crass2}.
     \begin{thm}\label{thm_form_aut}
        Define
        \begin{equation*}
            dX^I = (-1)^{\sigma_I}dx_{i_1} \wedge \cdots \wedge dx_{i_n},
        \end{equation*}
        where $I$ is the ordered set
        \begin{equation*}
            \{i_1,\ldots,i_n\} \quad i_1 < \cdots < i_n
        \end{equation*}
        and, for $\hat{i}$ the index not in $I$, $\sigma_I$ is the permutation
        \begin{equation*}
            \begin{pmatrix}
              0 &1&\cdots&n\\
              \hat{i} &i_1&\ldots & i_n
            \end{pmatrix}.
        \end{equation*}
        $\Gamma$ invariant $n$-forms
        \begin{equation*}
            \phi = \sum_{\hat{i}=0}^n f_{\hat{i}}dX^I
        \end{equation*}
        are in 1-1 correspondence with maps
        \begin{equation*}
            f = (f_0,\dots, f_n)
        \end{equation*}
        with $\Gamma \subset \Aut(f)$.
     \end{thm}
     \begin{proof}
        Hodge duality tells us that
        \begin{equation*}
            \sum_{i=1}^n f_i dx_i = f \cdot dX^I.
        \end{equation*}
        If $f^{\gamma} = \chi(\gamma)f$, then
        \begin{equation*}
            \sum_{i=1}^n (f^{\gamma})_i dx_i = \chi(\gamma)f \cdot dX^I.
        \end{equation*}
        The left-hand side is the same as $\gamma$ acting on
        \begin{equation*}
            \sum_{i=1}^n f_i(\gamma^{-1}\bar{x}) dx_i.
        \end{equation*}
        Recall that the action of $\Gamma$ on the exterior algebra is given by
        \begin{equation*}
            \gamma \phi(\bar{x}) = \phi(\gamma^{-1}\bar{x}).
        \end{equation*}
        So we have
        \begin{equation*}
            \gamma f \cdot dX^I = \chi(\gamma)f \cdot dX^I
        \end{equation*}
        and the $n$-form is $\Gamma$-invariant.

        Conversely, assume that the $n$-form $f \cdot dX^I$ is $\Gamma$-invariant, so that
        \begin{equation*}
            \gamma f \cdot dX^I = \chi(\gamma)f \cdot dX^I.
        \end{equation*}
        Again by Hodge duality, this is the same as saying that $\gamma$ acting on $\sum_{i=1}^n f_i(\gamma^{-1}\bar{x}) dx_i$ is the same as $\chi(\gamma)\sum_{i=1}^n fdx_i$. Hence, the map $f$ is $\Gamma$-equivariant.
     \end{proof}

    \begin{cor}
        Let $\Gamma$ be a finite subgroup of $\PGL_{N+1}$. Then there is a rational map $f:\P^N \to \P^N$ such that $\Gamma \subseteq \Aut(f)$.
    \end{cor}
    \begin{proof}
        We know that there are at least $N+1$ algebraically independent (primary) invariants for $\Gamma$, $p_0,\ldots,p_{N}$. So the $(N+1)$-form
        \begin{equation*}
            dp_0 \wedge \cdots \wedge dp_{N}
        \end{equation*}
        is $\Gamma$-invariant. Applying the previous theorem, this $(N+1)$-form corresponds to an $f$ with $\Gamma \subseteq \Aut(f)$.
    \end{proof}
    \begin{rem}
        Note that the previous corollary did not rule out the possibility that $f$ is the identity map as a projective map. In particular, the situation where $f$ is some multiple of the identity map, $f = (Fx_0,\ldots, Fx_N)$ for some homogeneous polynomial $F$ is still possible. While we could probably prove directly that this is not possible using the linear independence of the invariants $p_i$, it is easier to use the module of equivariants to prove the stronger statement of Theorem \ref{thm_exist_aut}.
    \end{rem}

    Using the module of equivariants, we can prove something stronger although somewhat less constructive. We say less constructive, since while there is a Reynold's operator for equivariants, it is not practical as the dimension increases. It is somewhat easier to find the invariants of $\Gamma$ and then construct the $(N+1)$-form as in the previous corollary.

    \begin{thm} \label{thm_exist_aut}
        Let $\Gamma$ be a finite subgroup of $\PGL_{N+1}$. Then there are infinitely many morphisms $f:\P^N \to \P^N$ such that $\Gamma \subseteq \Aut(f)$.
    \end{thm}
    \begin{proof}
        We can compute the number of fundamental equivariants $m$ from Proposition \ref{prop_fund_equi}. Since maps on $\P^N$ have coordinates in $K[x_0,\ldots, X_{N}]$, we have $m \geq N+1 \geq 2$; so for non-pseudoreflection groups \cite{Shephard-Todd}, which are all of our groups, the number of fundamental equivariants is at least $4$. In particular, there is at least one nontrivial equivariant $f$. Assume that $f$ is the identity map on projective space, i.e., $f = (Fx_0,\ldots,Fx_N)$ for some homogeneous polynomial $F$. This is an element in the module of equivariants so that $F$ must be an invariant of $\Gamma$. However, this equivariant is not independent of the trivial equivariant, contradicting the fact that $m \geq 4$.

        We now know there is at least one equivariant that is not the trivial map, call it $G = [G_0,\ldots,G_N]$. Let $p_1,\ldots, p_N$ be primary invariants for $\Gamma$. Since the equivariants are a module over the ring $K[p_1,\ldots,p_N]$, we can form new equivariants as
        \begin{equation*}
            F = \sum t_i G_i,
        \end{equation*}
        where $t_i \in K[p_1,\ldots,p_N]$, the $G_i$ are equivariants, and the degrees $\deg(t_iG_i)$ are all the same. Each such map can be thought of as a point in some affine space $\A^{\tau}$. The identification is between the coefficients of the $p_i$ in each $t_i$ with the affine coordinates. We have $\tau \geq 1$ since we can find at least one pair of equivariants $G_0,G_1$ whose degrees are such that we can create a homogeneous map $t_0G_0 + t_1G_1$. This is clear by looking at the form of the degrees of the fundamental equivariants from Proposition \ref{prop_fund_equi}.

        Recall that the map $F$ is a morphism if and only if the Macaulay resultant is non-zero and that the Macaulay resultant is a polynomial in the coefficients of the map (i.e., a closed condition). Thus, an open set in $\A^{\tau}$ corresponds to new equivariants.
    \end{proof}

    \begin{exmp}
        We consider the octahedral group as in Example \ref{exmp.equi.oct}. Using the fundamental equivariants
        \begin{align*}
            f_5(x,y)=&(-x^5 + 5xy^4 : 5x^4y - y^5)\\
            f_{17}(x,y)=&(x^{17} - 60x^{13}y^4 + 110x^9y^8 + 212x^5y^{12} - 7xy^{16}\\ &:-7x^{16}y + 212x^{12}y^5 + 110x^8y^9 - 60x^4y^{13} + y^{17})
        \end{align*}
        and the invariants
        \begin{align*}
            p_8=&x^8 + 14x^4y^4 + y^8\\
            p_{12}=&x^{10}y^2 - 2x^6y^6 + x^2y^{10} = (x^5y - xy^5)^2
        \end{align*}
        we construct a new equivariant:
        \begin{align*}
            f_{17} + 2p_{12}f_5 &=\\
            &(x^{17} + 2x^{15}y^2 - 60x^{13}y^4 - 14x^{11}y^6 + 110x^9y^8 +22x^7y^{10} + 212x^5y^{12} - 10x^3y^{14} - 7xy^{16}\\
            &: -7x^{16}y - 10x^{14}y^3 + 212x^{12}y^5 + 22x^{10}y^7 + 110x^8y^9 - 14x^6y^{11} -60x^4y^{13} + 2x^2y^{15} + y^{17}).
        \end{align*}
        Generalizing this to
        \begin{equation*}
            g_{t} = f_{17} + tp_{12}f_5,
        \end{equation*}
        we compute the Macaulay resultant as
        \begin{equation*}
            \Res(g_{t}) = 7784569084133195610168304788058590337455626404233216(t - 1)^6(t- 4/3)^{16}.
        \end{equation*}
        So for any choice of $t$ except $1$ and $4/3$, we produce a new equivariant morphism for the octahedral group.
\begin{code}
\begin{python}
K.<w>=CyclotomicField(8)
R.<t>=PolynomialRing(K)
P.<x,y>=ProjectiveSpace(R,1)
H=End(P)
f7=H([7*x^4*y^3 + y^7 , -x^7 - 7*x^3*y^4])
f5=H([-x^5 + 5*x*y^4, 5*x^4*y - y^5])
f17 = H([x^17 - 60*x^13*y^4 + 110*x^9*y^8 + 212*x^5*y^12 - 7*x*y^16,-7*x^16*y + 212*x^12*y^5 + 110*x^8*y^9 - 60*x^4*y^13 + y^17])
p6 = x^5*y - x*y^5
p8 = x^8 + 14*x^4*y^4+y^8
p12 = x^12 - 33*x^8*y^4 - 33*x^4*y^8 + y^12
M1=matrix(K,2,[(-1+w^2)/2,(-1+w^2)/2, (1+w^2)/2,(-1-w^2)/2])
a=w^3-w
M2=matrix(K,2,[(1+w^2)/a,0, 0,(1-w^2)/a])

f = H([f17[0] + t*f5[0]*p12,f17[1] + t*f5[1]*p12])
f.conjugate(M1)==f, f.conjugate(M2)==f

f.resultant().factor()
\end{python}
\end{code}
    \end{exmp}

\subsection{Exact Automorphism Groups and No Tetrahedral over $\Q$} \label{sect_no_tetra}
    Theorem \ref{thm_exist_aut} tells us that every finite subgroup of $\PGL_2$ occurs as a subgroup of an automorphism group. In this section, we demonstrate something stronger: for every finite $\Gamma \subset \PGL_2$ we can find a $f:\P^1 \to \P^1$ with $\Gamma = \Aut(f)$. Each of these maps is defined over $\Q$ except the one with tetrahedral automorphism group. We show that it is, in fact, not possible to have a $\Q$-rational map with tetrahedral automorphism group.

    Constructing the examples is mainly a computational problem using the Molien series and the Reynolds operator.
    As an example, say we want to construct a map with $\Aut(f) = A_4$ (tetrahedral). Then we can compute the Molien series for all characters of $A_4$ and $S_4$ (octahedral) as
    \begin{align*}
    A_4:&\chi_1:1 + t^8 + 2t^{12} + t^{16} + O(t^{20})\\
    &\chi_2:t^6 + t^{14} + 2t^{18} + t^{22} + O(t^{26})\\
    &\chi_3:t^{10} + t^{14} + t^{18} + 2t^{22} + 2t^{26} + O(t^{30})\\
    &\chi_4:t^4 + t^8 + t^{12} + 2t^{16} + 2t^{20} + O(t^{24})\\
    &\chi_5:t^4 + t^8 + t^{12} + 2t^{16} + 2t^{20} + O(t^{24})\\
    &\chi_6:t^{10} + t^{14} + t^{18} + 2t^{22} + 2t^{26} + O(t^{30})\\
    S_4:&\chi_1:1 + t^8 + t^{12} + t^{16} + t^{18} + O(t^{20})\\
    &\chi_2:t^6 + t^{12} + t^{14} + t^{18} + t^{20} + t^{22} + t^{24} + O(t^{26}).
    \end{align*}
    We are looking for invariants of $A_4$ that are not invariants of $S_4$. We see that there should be a relative invariant of degree $4$ that is not an invariant of $S_4$. Using the Reynolds operator, we have the invariant
    \begin{equation*}
        x^4 + 2\sqrt{-3}x^2y^2 + y^4.
    \end{equation*}
    Via Klein (\ref{eq_Klein}), this yields the map
    \begin{align*}
        f&:\P^1 \to \P^1\\
        f(x,y) &= (\sqrt{-3}x^2y - y^3, x^3 + \sqrt{-3}xy^2).
    \end{align*}
    We can then verify that the automorphism group is exactly $A_4$.
    We can similarly construct a map for each of the other finite subgroups of $\PGL_2$. The list is given in Figure \ref{fig:exact}.
    \begin{figure}
    \caption{Exact Automorphism Groups}
    \label{fig:exact}
    \begin{align*}
        C_n: &f(x,y) = (x^{n+1} + xy^n : y^{n+1})\\
        D_{2n}: &f(x,y) = (y^{n-1} : x^{n-1})\\
        A_4: &f(x,y) = (\sqrt{-3}x^2y - y^3 : x^3 + \sqrt{-3}xy^2)\\
        S_4: &f(x,y) = ( -x^5 + 5xy^4: 5x^4y - y^5)\\
        A_5: &f(x,y) = (-(x^{11} + 66x^6y^5 - 11xy^{10}) : 11x^{10}y + 66x^5y^6 - y^{11})
    \end{align*}
    \end{figure}
    Notice that each of these examples is defined over $\Q$ except for $A_4$ (tetrahedral).
    This leads to the question as to whether any map defined over $\Q$ has a tetrahedral automorphism group and to the more general question: How does the field of definition affect the existence of automorphisms?

    We show next that the tetrahedral group does not occur as the automorphism group of any $f:\P^1 \to \P^1$ defined over $\Q$ and do not address the more general question.

    \begin{thm}\label{thm_no_tetra}
        There is no morphism $f:\P^1 \to \P^1$ defined over $\Q$ which has tetrahedral group (as represented above) as automorphism group.
    \end{thm}
    \begin{proof}
        By Blichfeldt \cite[\S 104-105]{Blichfeldt2} every invariant $F$ of the tetrahedral group can be written as a product of powers of the following three invariants
        \begin{align*}
            t_1=&x^4 + 2\sqrt{-3}x^2y^2 + y^4\\
            t_2=&x^4 - 2\sqrt{-3}x^2y^2 + y^4\\
            t_3=&xy(x^4-y^4).
        \end{align*}
        Similarly, invariants for the octahedral group can be constructed from
        \begin{align*}
            s_1=&xy(x^4-y^4)\\
            s_2=&x^8 + 14x^4y^4 + y^8 = t_1t_2\\
            s_3=&x^{12} - 33x^8y^4 - 33x^4y^8 + y^{12}.
        \end{align*}
        Note that $s_2 = t_1t_2$.

        For a map of Klein's form (\ref{eq_Klein}) to be tetrahedral and defined over $\Q$, its invariant must be of the form $t_1^{a_1}t_2^{a_1}t_3^{a_3}$. However, this is the same as $s_2^{a_1}s_1^{a_2}$ so it will have octahedral symmetries.

        We also need to consider maps that come from a nontrivial $F,G$ pair for the Doyle-McMullen construction of Proposition \ref{prop_DM} with at least one not defined over $\Q$. Actually, if one of $F$ or $G$ is not defined over $\Q$, then to end up with a map defined over $\Q$ we must have both not defined over $\Q$. Assume that $F$ has a term of the form $c x^ny^m$ with $c \not\in \Q$. We are constructing the coordinates of the map as $xF/2 + G_y$ and $yF/2 - G_x$. In the first coordinate, we must have a monomial $\frac{cx^{n+1}y^m}{2}$ coming from $xF/2$, so $G$ must have a term $-\frac{cx^{n+1}y^{m+1}}{2(m+1)}$ for the map to be defined over $\Q$. Similarly, from the second coordinate we see that $G$ has a term $\frac{cx^{n+1}y^{m+1}}{2(n+1)}$. We conclude that $c=0$.
    \end{proof}

\begin{code}
\begin{verbatim}
R<x> := PolynomialRing(Integers());
f := x^2+1;
K<i> := NumberField(f);
R<x> := PolynomialRing(K);
f := x^2-2;
L<a> := ext<K|f>;
R<x> := PolynomialRing(L);
f := x^2+3;
L<b> := ext<L|f>;

Oct := MatrixGroup< 2, L | [(-1+i)/2,(-1+i)/2, (1+i)/2,(-1-i)/2], [(1+i)/a,0, 0,(1-i)/a] >;
Tet := MatrixGroup< 2, L | [(-1+i)/2,(-1+i)/2, (1+i)/2,(-1-i)/2], [0,i, -i,0] >;
Di8 := MatrixGroup< 2, L | [0,1,1,0], [i,0,0,-i]>;

PS<t>:=PowerSeriesRing(L, 40);
MOct:= PS!MolienSeries(Oct);
MTet:= PS!MolienSeries(Tet);
MDi8 :=PS!MolienSeries(Di8);

//This needs to also divide by relative invariants to work...
MTet/MOct;
MolienSeries(Tet)/MolienSeries(Oct);


//Compute the Molien series for linear characters.



L<w> := CyclotomicField(24);
i:=w^6;
a:= - w^5 + w^3 + w;
//Tetrahedral
H := MatrixGroup< 2, L| [(-1+i)/2,(-1+i)/2, (1+i)/2,(-1-i)/2], [0,i, -i,0]  >;
//Octahedral
//H:= MatrixGroup< 2, L | [(-1+i)/2,(-1+i)/2, (1+i)/2,(-1-i)/2], [(1+i)/a,0,0,(1-i)/a]  >;

CT:=CharacterTable(H);  //first 6 have no zeros

for j:=1 to 6 do
  xi:=CT[j];
  S:=Inverse(NumberingMap(H));
  R<t>:=PolynomialRing(L);
  Mol:=0;
  for i:=1 to Order(H) do
    Mol := Mol + (R!xi(S(i)))/(Determinant(DiagonalMatrix([1,1])- t*Matrix(R,S(i))));
  end for;
  Mol:=Mol/Order(H);
  PS<t>:=PowerSeriesRing(L);
  PS!Mol;
end for;

//Apply the Reynolds operator
L<w> := CyclotomicField(24);
i:=w^6;
a:= - w^5 + w^3 + w;
//Tetrahedral
Tet := MatrixGroup< 2, L| [(-1+i)/2,(-1+i)/2, (1+i)/2,(-1-i)/2], [0,i, -i,0]  >;

R<x,y>:=PolynomialRing(L,2);
X:=Matrix(R,2,1,[x,y]);
M:=MonomialsOfDegree(R,4);
S:=Inverse(NumberingMap(Tet));
CT:=CharacterTable(Tet);
xi:=CT[4];
INV:={};

for i:=1 to #M do
  T:=0;
  F:=M[i];
  for i:= 1 to #Tet do
    t:=Matrix(R,S(i))*X;
    G:=(L!xi(S(i)))*Evaluate(F,[t[1,1],t[2,1]]);
    T:=T+G;
  end for;
  INV:=Include(INV, T);
end for;
INV;
\end{verbatim}

\begin{verbatim}
//determine if the group is in the automorphism group
R<x> := PolynomialRing(Integers());
f := x^2+1;
K<i> := NumberField(f);
R<x> := PolynomialRing(K);
f := x^2-2;
L<a> := ext<K|f>;
R<x> := PolynomialRing(L);
f := x^2+3;
L<b> := ext<L|f>;

L<w> := CyclotomicField(24);
i:=w^6;
a:= - w^5 + w^3 + w;

Oct := MatrixGroup< 2, L | [(-1+i)/2,(-1+i)/2, (1+i)/2,(-1-i)/2], [(1+i)/a,0, 0,(1-i)/a]  >;
Tet := MatrixGroup< 2, L | [(-1+i)/2,(-1+i)/2, (1+i)/2,(-1-i)/2], [0,i, -i,0]  >;
Di8 := MatrixGroup< 2, L | [0,1,1,0], [i,0,0,-i]>;

K<z5>:=CyclotomicField(5);
R<x> := PolynomialRing(K);
f:=x^2+1;
L1<i>:=ext<K | f>;
R<x,y> := PolynomialRing(L1,2);
f:=x^2-2;
L<b>:=ext<L1 | f>;
R<x,y> := PolynomialRing(L,2);


X:=Matrix(R,2,1,[x,y]);
a:=2*z5^3 + 2*z5^2 + 1; //sqrt(5)
Icos := MatrixGroup< 2, L| [z5^3,0, 0,z5^2], [0,1, -1,0], [(z5^4-z5)/a, (z5^2-z5^3)/a, (z5^2-z5^3)/a, -(z5^4-z5)/a]  >;


R<x> := PolynomialRing(Integers());
f := x^2+1;
K<i> := NumberField(f);
R<x> := PolynomialRing(K);
f:=x^2+3;
L<b>:=ext<K | f>;
R<x,y> := PolynomialRing(L,2);

Tet := MatrixGroup< 2, L | [(-1+i)/2,(-1+i)/2, (1+i)/2,(-1-i)/2], [0,i, -i,0]  >;
H:=Tet;
R<x,y>:=PolynomialRing(L,2);
X:=Matrix(R,2,1,[x,y]);
//F:=Matrix(R,2,1,[11*x^8*y^3 + 22*x^4*y^7 - y^11, x^11 - 22*x^7*y^4 - 11*x^3*y^8]); //map
//F:=Matrix(R,2,1,[-4*y^3 - 2*(4*w^4-2)*x^2*y, 2*(4*w^4-2)*x*y^2 + 4*x^3]); //tet
F:=Matrix(R,2,1,[-x^5 + 5*x*y^4, 5*x^4*y - y^5]); //Oct
F:=Matrix(R,2,1,[-(x^11 + 66*x^6*y^5 - 11*x*y^10), 11*x^10*y + 66*x^5*y^6 - y^11]);  //icos

F:=Matrix(R,2,1,[-b*x^2*y - y^3, x^3 + b*x*y^2]);

Auts:={};
S:=Inverse(NumberingMap(H));
for i:=1 to Order(H) do
    t:=Matrix(R,S(i))*X;
    G:=Matrix(R,2,1,[Evaluate(F[1,1],[t[1,1],t[2,1]]),Evaluate(F[2,1],[t[1,1],t[2,1]])]);
    C:=Matrix(R,S(i))^(-1)*G;
    if C[1,1]/C[2,1] eq F[1,1]/F[2,1] then
        Auts:=Include(Auts, S(i));
    end if;
end for;
#Auts;

\end{verbatim}
\end{code}

\begin{code}
\begin{verbatim}
L<w> := CyclotomicField(24);
i:=w^6;
a:= - w^5 + w^3 + w;
//Tetrahedral
H := MatrixGroup< 2, L| [(-1+i)/2,(-1+i)/2, (1+i)/2,(-1-i)/2], [0,i, -i,0]  >;

CT:=CharacterTable(H);  //first 6 have no zeros
MTet:=0;
for j:=1 to 6 do
  xi:=CT[j];
  S:=Inverse(NumberingMap(H));
  R<t>:=PolynomialRing(L);
  Mol:=0;
  for i:=1 to Order(H) do
    Mol := Mol + (R!xi(S(i)))/(Determinant(DiagonalMatrix([1,1])- t*Matrix(R,S(i))));
  end for;
  Mol:=Mol/Order(H);
  PS<t>:=PowerSeriesRing(L,100);
  MTet := MTet+PS!Mol;
end for;

//Octahedral
H:= MatrixGroup< 2, L | [(-1+i)/2,(-1+i)/2, (1+i)/2,(-1-i)/2], [(1+i)/a,0,0,(1-i)/a]  >;


CT:=CharacterTable(H);  //first 6 have no zeros
MOct:=0;
for j:=1 to 6 do
  xi:=CT[j];
  S:=Inverse(NumberingMap(H));
  R<t>:=PolynomialRing(L);
  Mol:=0;
  for i:=1 to Order(H) do
    Mol := Mol + (R!xi(S(i)))/(Determinant(DiagonalMatrix([1,1])- t*Matrix(R,S(i))));
  end for;
  Mol:=Mol/Order(H);
  PS<t>:=PowerSeriesRing(L,100);
  print Mol;
  MOct := MOct+PS!Mol;
end for;
MTet/MOct;
\end{verbatim}

\begin{verbatim}
L<w> := CyclotomicField(24);
i:=w^6;
a:= - w^5 + w^3 + w;

//Tetrahedral
H := MatrixGroup< 2, L| [(-1+i)/2,(-1+i)/2, (1+i)/2,(-1-i)/2], [0,i, -i,0]  >;

//Octahedral
//H:= MatrixGroup< 2, L | [(-1+i)/2,(-1+i)/2, (1+i)/2,(-1-i)/2], [(1+i)/a,0,0,(1-i)/a]  >;


//Apply the equivariant Reynolds operator
R<x,y> := PolynomialRing(L,2);
X:=Matrix(R,2,1,[x,y]);
M:=MonomialsOfDegree(R,7);
Auts:={};


CT:=CharacterTable(H);
xi:=CT[3];
S:=Inverse(NumberingMap(H));
for xxi:=1 to #M do
  for yi:=1 to #M do
      T:=Matrix(2,1,[0,0]);
      F:=Matrix(R,2,1,[M[xxi],M[yi]]);
      for i:= 1 to #H do
        t:=Matrix(R,S(i))*X;
        G:=Matrix(R,2,1,[Evaluate(F[1,1],[t[1,1],t[2,1]]),Evaluate(F[2,1],[t[1,1],t[2,1]])]);
        C:=Matrix(R,S(i))^(-1)*G;
        T:=T+R!L!xi(S(i))*C;
      end for;
      g:=GCD([T[1,1],T[2,1]]);
      if g ne 0 then
        T[1,1]:=T[1,1]/g;
        T[2,1]:=T[2,1]/g;
      end if;
      Auts:=Include(Auts, T);
    end for;
  end for;
Auts;
\end{verbatim}
\end{code}

\section{Finite Subgroups of $\PGL_{3}$} \label{sect.P2}
     For $N=1$, $2$, and $3$ there are classical descriptions of the finite subgroups of $\PGL_{N+1}$. The generators for $\PGL_2$ subgroups as elements of $\SL_2$ was provided by Silverman \cite{Silverman12} and restated above in Section \ref{sect.4.1}. In this section, we provide generators of the finite subgroups of $\PGL_3$ as elements of $\SL_{3}$ using the classical description from Blichfeldt \cite[Chapters III and V]{Blichfeldt}.  We denote $\zeta_n$ as a primitive $n$-th root of unity. In $\PGL_3$ we have the additional complication that there are multiple inequivalent representations of the finite groups.  These representations can have different behavior with respect to the existence of automorphisms.

     \begin{exmp}
        We can find two inequivalent representations of $C_5$ as
        \begin{align*}
            C_5 = \left \langle \begin{pmatrix}
              \zeta_5&0&0 \\ 0&1&0 \\ 0&0&\zeta_5^{-1}
            \end{pmatrix} \right \rangle\\
            C_5' = \left \langle \begin{pmatrix}
              \zeta_5&0&0 \\ 0&\zeta_5^2&0 \\ 0&0&\zeta_5^2
            \end{pmatrix}\right \rangle
        \end{align*}
        We have that the morphism $f$ has $C_5 \subset \Aut(f)$ for
        \begin{equation*}
            f(x,y,z) = [z^4,y^4,x^4].
        \end{equation*}
        However, there is no morphism $f: \P^2 \to \P^2$ of degree~$4$ such that  $C_5' \subseteq \Aut(f)$.

        \begin{code}
\begin{verbatim}
K<a>:=CyclotomicField(5);
H := MatrixGroup< 3, K| [a,0,0, 0,a^2,0, 0,0,a^2] >;
H := MatrixGroup< 3, K| [a,0,0, 0,1,0, 0,0,a^4]>;
S:=Inverse(NumberingMap(H));
R<x,y,z>:=PolynomialRing(K,3);
X:=Matrix(R,3,1,[x,y,z]);
M:=MonomialsOfDegree(R,4);
Auts:={};

for xi:=1 to #M do
  for yi:=1 to #M do
    for zi:=1 to #M do
        T:=Matrix(3,1,[0,0,0]);
        F:=Matrix(R,3,1,[M[xi],M[yi],M[zi]]);
      for i:= 1 to #H do
        t:=Matrix(R,S(i))*X;
        G:=Matrix(R,3,1,[Evaluate(F[1,1],[t[1,1],t[2,1],t[3,1]]),Evaluate(F[2,1],[t[1,1],t[2,1],t[3,1]]),Evaluate(F[3,1],[t[1,1],t[2,1],t[3,1]])]);
        C:=Matrix(R,S(i))^(-1)*G;
        T:=T+C;
      end for;
      g:=GCD([T[1,1],T[2,1],T[3,1]]);
     if g ne 0 then
        T[1,1]:=T[1,1]/g;
        T[2,1]:=T[2,1]/g;
        T[3,1]:=T[3,1]/g;
      end if;
      Auts:=Include(Auts, T);
    end for;
  end for;
end for;
Auts;
\end{verbatim}
\end{code}
     \end{exmp}

\begin{code}
\begin{verbatim}
A:=CharacterTable(H);
<IsFaithful(A[i]) : i in {1..#A}>;
CharacterDegrees(H);
\end{verbatim}
\end{code}

\begin{code}
\begin{verbatim}
T:=InvariantRing(H);
PS<t>:=PowerSeriesRing(L);
MolienSeries(H);
PS!MolienSeries(H);
FundamentalInvariants(T);


A:=InvariantsOfDegree(T,6)[1];
B:=InvariantsOfDegree(T,12)[1];
AssignNames(~T,["x","y","z"]);
Ax:=Derivative(A,1);
Ay:=Derivative(A,2);
Az:=Derivative(A,3);
Bx:=Derivative(B,1);
By:=Derivative(B,2);
Bz:=Derivative(B,3);

F:=Matrix(R,3,1,[Ay*Bz-Az*By,Az*Bx-Ax*Bz,Ax*By-Ay*Bx]);
r:=GCD([F[1,1],F[2,1],F[3,1]]);
F[1,1]:=F[1,1]/r;
F[2,1]:=F[2,1]/r;
F[3,1]:=F[3,1]/r;

t:=Matrix(R,H.1)*X;
G:=Matrix(R,3,1,[Evaluate(F[1,1],[t[1,1],t[2,1],t[3,1]]),Evaluate(F[2,1],[t[1,1],t[2,1],t[3,1]]),Evaluate(F[3,1],[t[1,1],t[2,1],t[3,1]])]);
C:=Matrix(R,H.1)^(-1)*G;
C eq F;
\end{verbatim}
\end{code}

    \begin{rem}
        The number of inequivalent representation of $C_n$ tends to infinity as $n$ tends to infinity.
    \end{rem}

    \subsubsection{Intransitive and Imprimitive Groups}
    \begin{enumerate}
        \item[(A)] Cyclic Group of order $n$, $C_n$ generated by
                \[\begin{pmatrix}
                    \zeta_n^a&0&0\\0&\zeta_n^b&0\\0&0&1
                    \end{pmatrix}
                \]
                where $\gcd(a,n) = 1$ or $\gcd(b,n) = 1$.

        \item[(B)] Subgroups of the form
            \[\begin{pmatrix} \zeta_p&0&0\\0&a&b \\0&c&d \end{pmatrix},\]
            where the lower right $2 \times 2$ transformation is one of the $\PGL_2$ transformations below. In this case we use the representations of Blichfeldt \cite[Chapters III]{Blichfeldt} instead of Silverman \cite{Silverman12} to minimize the order of the representation. In particular, a representation of a $\PGL_2$ subgroup of order $g$ can have order $gt$ in $\GL_2$ for some $t$. We chose a representation of this form that minimizes $t$. Note that it is possible to represent the symmetries of the solids as dimension 3 representations of exact order ($t=1$).
            \begin{enumerate}
                \item[(B1)] Dihedral of order $2q$:
                    \[
                    \left\langle
                    \begin{pmatrix} \zeta_q &0\\0&1 \end{pmatrix}, \quad \begin{pmatrix} 0 & 1\\ 1 & 0 	 \end{pmatrix}\right\rangle.
                     \]

\begin{code}
\begin{verbatim}
L<w>:=CyclotomicField(12);
R<x,y,z> := PolynomialRing(L,3);
X:=Matrix(R,3,1,[x,y,z]);

a := Matrix(L,3,[w^4,0,0, 0,1,0, 0,0,1]);
b := Matrix(L,3,[1,0,0, 0,w^2,0, 0,0,1/w^2]);
c := Matrix(L,3,[1,0,0, 0,0,1, 0,1,0]);
H := MatrixGroup< 3, L| [a,b,c] >;




phiH:=Inverse(NumberingMap(H));
MR:=Parent(a);
LH:=[];
for i:=1 to #H do
y:=1/phiH(i)[1,1]*MR!(phiH(i));
if not y in LH then
Append(~LH,y);
end if;
end for;
S:=CyclicSubgroups(H);
for T in S do
G:=T`subgroup;
phiG:=Inverse(NumberingMap(G));
LG:=[];
for i:=1 to #G do
y:=1/phiG(i)[1,1]*MR!(phiG(i));
if not y in LG then
Append(~LG,y);
end if;
end for;
print #LH,#LG;
end for;

FPGroup(H2);

for j:=1 to 20 do
  if c*a^j eq a^j*c then
    print j;
  end if;
end for;

H := MatrixGroup< 3, L| [w^3,0,0, 0,w^2,0, 0,0,w^(-5)],[1,0,0, 0,0,1, 0,1,0] >;
s := Matrix(R,3,[w^3,0,0, 0,w^2,0, 0,0,w^(-5)]);
t := Matrix(R,3,[1,0,0, 0,0,1, 0,1,0]);

H:=MatrixGroup<3,L | [a,b]>;

s := Matrix(R,3,[w,0,0, 0,w^2,0, 0,0,w^(-2)]);
t := Matrix(R,3,[1,0,0, 0,0,1, 0,1,0]);
H := MatrixGroup< 3, L| [s,t] >;


for j:=1 to 20 do
  if t*s^j eq s^j*t then
    print j;
  end if;
end for;


FPGroup(H2);


a := Matrix(L,2,[w^2,0, 0,1]);
b := Matrix(L,2,[0,1, 1,0]);
H := MatrixGroup< 2, L| [a,b] >;

phiH:=Inverse(NumberingMap(H));
MR:=Parent(a);
LH:=[];
for i:=1 to #H do
  z:=MR!(phiH(i));
  if z[1,1] ne 0 then
    y:=1/phiH(i)[1,1]*z;
  else
    y:=1/phiH(i)[1,2]*z;
  end if;
  if not y in LH then
    Append(~LH,y);
  end if;
end for;
S:=CyclicSubgroups(H);
for T in S do
  G:=T`subgroup;
  phiG:=Inverse(NumberingMap(G));
  LG:=[];
  for i:=1 to #G do
    z:=MR!(phiG(i));
    if z[1,1] ne 0 then
      y:=1/phiG(i)[1,1]*z;
    else
      y:=1/phiG(i)[1,2]*z;
    end if;
    if not y in LG then
      Append(~LG,y);
    end if;
  end for;
print #LH,#LG;
end for;
\end{verbatim}
\end{code}

                \item[(B2)] Tetrahedral of order $12$:

                    \begin{equation*}
                         \left\langle \frac{1}{2}\begin{pmatrix}
                          -1+i & -1+i \\ 1+i & -1-i
                        \end{pmatrix},
                        \quad \begin{pmatrix}
                          i & 0 \\ 0& -i
                        \end{pmatrix} \right\rangle.
                    \end{equation*}

\begin{code}
\begin{verbatim}
K<a>:=CyclotomicField(7);
R<x>:=PolynomialRing(K);
f:=x^2+1;
L1<i>:=ext<K | f>;
R<x> := PolynomialRing(L1);
f := x^2-2;
L<a> := ext<L1|f>;
R<x,y,z> := PolynomialRing(L,3);
X:=Matrix(R,3,1,[x,y,z]);

H := MatrixGroup< 3, L| [1,0,0, 0,1,0, 0,0,1], [1,0,0, 0,(-1+i)/2,(-1+i)/2, 0,(1+i)/2,(-1-i)/2], [1,0,0, 0,0,i, 0,-i,0]  >;


H := MatrixGroup< 3, L| [1,0,0, 0,1,0, 0,0,1], [1,0,0, 0,(-1+i)/2,(-1+i)/2, 0,(1+i)/2,(-1-i)/2], [1,0,0, 0,i,0, 0,0,-i]  >;

Subgroups(H: IsCyclic := true);


//3 dim representation of tetrahedral??
H2 := MatrixGroup< 3, L| [1,0,0, 0,-1,0, 0,0,-1], [-1,0,0, 0,1,0, 0,0,-1], [-1,0,0, 0,-1,0, 0,0,1], [0,0,1, 1,0,0, 0,1,0]  >;
\end{verbatim}
\end{code}
%

                \item[(B3)] Octahedral of order $24$:


                     \begin{equation*}
                      \left\langle \frac{1}{2}\begin{pmatrix}
                          -1+i & -1+i \\ 1+i & -1-i
                        \end{pmatrix},
                       \quad \frac{1}{\sqrt{2}}\begin{pmatrix} 1+i & 0\\ 0 & 1-i \end{pmatrix} \right\rangle.
                     \end{equation*}


\begin{code}
\begin{verbatim}
K<w>:=CyclotomicField(7);
R<x>:=PolynomialRing(K);
f:=x^2+1;
L1<i>:=ext<K | f>;
R<x> := PolynomialRing(L1);
f := x^2-2;
L<a> := ext<L1|f>;
R<x,y,z> := PolynomialRing(L,3);
X:=Matrix(R,3,1,[x,y,z]);

H := MatrixGroup< 3, L| [w^5,0,0, 0,w,0, 0,0,w],[1,0,0, 0,(-1+i)/2,(-1+i)/2, 0,(1+i)/2,(-1-i)/2], [1,0,0, 0,(1+i)/a,0, 0,0,(1-i)/a]  >;

//3 dim representation of octahedral??
H2:= MatrixGroup< 3, L| [0,1,0, 1,0,0, 0,0,1], [0,0,1, 1,0,0, 0,1,0], [-1,0,0, 0,1,0, 0,0,-1], [1,0,0, 0,-1,0, 0,0,-1]>;
\end{verbatim}
\end{code}

                \item[(B4)] Icosahedral of order $60$


                    \begin{equation*}
                        \left\langle \frac{1}{2}\begin{pmatrix}
                          -1+i & -1+i \\ 1+i & -1-i
                        \end{pmatrix}, \quad
                        \begin{pmatrix} i & 0 \\ 0& -i \end{pmatrix}, \quad
                        \begin{pmatrix} i/2 & \beta - i\gamma \\ -\beta - i\gamma& -i/2 \end{pmatrix}
                        \right\rangle.
                    \end{equation*}
                    where $\beta = \frac{1-\sqrt{5}}{4}$ and $\gamma = \frac{1+\sqrt{5}}{4}$.
\begin{code}
\begin{verbatim}
K<w>:=CyclotomicField(15);
R<x> := PolynomialRing(K);
f:=x^2+1;
L<i>:=ext<K | f>;
R<x,y,z> := PolynomialRing(L,3);
X:=Matrix(R,3,1,[x,y,z]);
z5:=w^3;  //\zeta_5
a:=2*z5^3 + 2*z5^2 + 1; //sqrt(5)
b := (1-a)/4;
g := (1+a)/4;

H := MatrixGroup< 3, L| [1,0,0, 0,1,0, 0,0,1], [1,0,0, 0,z5^3,0, 0,0,z5^2], [1,0,0, 0,0,1, 0,-1,0], [1,0,0, 0,(z5^4-z5)/a, (z5^2-z5^3)/a, 0,(z5^2-z5^3)/a, -(z5^4-z5)/a]  >;

H := MatrixGroup< 3, L| [1,0,0, 0,1,0, 0,0,1],[1,0,0, 0,(-1+i)/2,(-1+i)/2, 0,(1+i)/2,(-1-i)/2], [1,0,0, 0,i,0, 0,0,-i], [1,0,0, 0,i/2,b-i*g, 0, -b-i*g, -i/2]  >;


H := MatrixGroup< 3, L| [w,0,0, 0,1,0, 0,0,1], [1,0,0, 0,z5^3,0, 0,0,z5^2], [1,0,0, 0,0,1, 0,-1,0], [1,0,0, 0,(z5^4-z5)/a, (z5^2-z5^3)/a, 0,(z5^2-z5^3)/a, -(z5^4-z5)/a]  >;

a := Matrix(L,3,[w^8,0,0, 0,1,0, 0,0,1]);

phiH:=Inverse(NumberingMap(H));
MR:=Parent(a);
LH:=[];
for i:=1 to #H do
  z:=MR!(phiH(i));
  if z[1,1] ne 0 then
    y:=1/phiH(i)[1,1]*z;
  else
    y:=1/phiH(i)[1,2]*z;
  end if;
  if not y in LH then
    Append(~LH,y);
  end if;
end for;
S:=CyclicSubgroups(H);
for T in S do
  G:=T`subgroup;
  phiG:=Inverse(NumberingMap(G));
  LG:=[];
  for i:=1 to #G do
    z:=MR!(phiG(i));
    if z[1,1] ne 0 then
      y:=1/phiG(i)[1,1]*z;
    else
      y:=1/phiG(i)[1,2]*z;
    end if;
    if not y in LG then
      Append(~LG,y);
    end if;
  end for;
print #LH,#LG;
end for;


//The next two are from GAP:  IrreducibleRepresentations(AlternatingGroup(5));
w:=z5;
H1:=MatrixGroup<3,L|[ 1,-z5^2-z5^3,0, -z5-z5^4,z5^2+z5^3,-z5-z5^4, -1,0,0],[-1,-1,-1, -z5^2-z5^3,-z5^2-z5^3,z5+z5^4, -z5-z5^4,z5^2+z5^3,-z5-z5^4]>;

H2:=MatrixGroup<3,L|[2*z5+2*z5^4,-2*z5-z5^2-z5^3-2*z5^4,z5+z5^4, 2*z5+z5^2+z5^3+2*z5^4,-2*z5-z5^2-z5^3-2*z5^4,z5+z5^4, -z5-z5^4,-1,0],[0,0,1, z5+z5^4,1,0, -1,0,-1]>;
\end{verbatim}
\end{code}

            \end{enumerate}
    \end{enumerate}

    \subsubsection{Imprimitive Groups}
        \begin{enumerate}
            \item[(C)] The group generated by
                \begin{equation*}
                    \left \langle \begin{pmatrix}
                    \zeta_n^a&0&0\\0&\zeta_n^b&0\\0&0&1
                    \end{pmatrix}, \begin{pmatrix} 0&1&0\\0&0&1\\1&0&0 \end{pmatrix}\right \rangle.
                \end{equation*}

\begin{code}
\begin{verbatim}
L<w>:=CyclotomicField(4*9);
R<x,y,z> := PolynomialRing(L,3);
X:=Matrix(R,3,1,[x,y,z]);

H := MatrixGroup< 3, L| [w,0,0, 0,w^(6),0, 0,0,1],[0,1,0, 0,0,1, 1,0,0] >;
Factorization(#H);

a := Matrix(L,3,[w,0,0, 0,w^2,0, 0,0,1]);
b := Matrix(L,3,[0,1,0, 0,0,1, 1,0,0]);
phiH:=Inverse(NumberingMap(H));
MR:=Parent(a);
LH:=[];
for i:=1 to #H do
  z:=MR!(phiH(i));
  if z[1,1] ne 0 then
    y:=1/phiH(i)[1,1]*z;
  elif z[1,2] ne 0 then
    y:=1/phiH(i)[1,2]*z;
  else
    y:=1/phiH(i)[1,3]*z;
  end if;
  if not y in LH then
    Append(~LH,y);
  end if;
end for;
S:=CyclicSubgroups(H);
for T in S do
  G:=T`subgroup;
  phiG:=Inverse(NumberingMap(G));
  LG:=[];
  for i:=1 to #G do
    z:=MR!(phiG(i));
    if z[1,1] ne 0 then
      y:=1/phiG(i)[1,1]*z;
    elif z[1,2] ne 0 then
      y:=1/phiG(i)[1,2]*z;
    else
      y:=1/phiG(i)[1,3]*z;
    end if;
    if not y in LG then
      Append(~LG,y);
    end if;
  end for;
print #LH,#LG;
end for;


SG:=Subgroups(H: IsCyclic := true);
Order(H)/SG[#SG]`order;
\end{verbatim}
\end{code}
            \item[(D)] The group generated by
                \begin{equation*}
                    \left \langle \begin{pmatrix}
                    \zeta_n^a&0&0\\0&\zeta_n^b&0\\0&0&1
                    \end{pmatrix}, \begin{pmatrix} 0&1&0\\0&0&1\\1&0&0 \end{pmatrix},
                    \begin{pmatrix}
                        \zeta_n^x&0&0\\0&0&\zeta_n^y\\0&1&0
                    \end{pmatrix} \right \rangle.
                \end{equation*}
\begin{code}
\begin{verbatim}
L<w>:=CyclotomicField(28);
R<x,y,z> := PolynomialRing(L,3);
X:=Matrix(R,3,1,[x,y,z]);
z:=w^7;
u:=w^4;
v:=w^8;
t:=1;


H := MatrixGroup< 3, L| [z,0,0, 0,z^2,0, 0,0,1],[0,1,0, 0,0,1, 1,0,0],[u,0,0, 0,0,v, 0,1,0] >;
Factorization(#H);

H2 := MatrixGroup< 3, L| [0,1,0, 0,0,1, 1,0,0],[1,0,0, 0,0,1, 0,1,0] >;

a := Matrix(L,3,[w,0,0, 0,1,0, 0,0,1]);
phiH:=Inverse(NumberingMap(H));
MR:=Parent(a);
LH:=[];
for i:=1 to #H do
  z:=MR!(phiH(i));
  if z[1,1] ne 0 then
    y:=1/phiH(i)[1,1]*z;
  elif z[1,2] ne 0 then
    y:=1/phiH(i)[1,2]*z;
  else
    y:=1/phiH(i)[1,3]*z;
  end if;
  if not y in LH then
    Append(~LH,y);
  end if;
end for;
S:=CyclicSubgroups(H);
for T in S do
  G:=T`subgroup;
  phiG:=Inverse(NumberingMap(G));
  LG:=[];
  for i:=1 to #G do
    z:=MR!(phiG(i));
    if z[1,1] ne 0 then
      y:=1/phiG(i)[1,1]*z;
    elif z[1,2] ne 0 then
      y:=1/phiG(i)[1,2]*z;
    else
      y:=1/phiG(i)[1,3]*z;
    end if;
    if not y in LG then
      Append(~LG,y);
    end if;
  end for;
print #LH,#LG;
end for;

SG:=Subgroups(H: IsCyclic := true);
Order(H)/SG[#SG]`order;
\end{verbatim}
\end{code}

        \end{enumerate}

    \subsubsection{Primitive Groups having normal imprimitive subgroups}
        \begin{enumerate}
            \item[(E)] The group of order 36 generated by
                \begin{equation*}
                    \left \langle \begin{pmatrix} \zeta_3&0&0\\0&\zeta_3^2&0\\0&0&1 \end{pmatrix},
                    \begin{pmatrix} 0&1&0\\0&0&1\\1&0&0 \end{pmatrix},
                    \frac{1}{\zeta_3 - \zeta_3^2}\begin{pmatrix} 1&1&1\\1&\zeta_3&\zeta_3^2\\1&\zeta_3^2&\zeta_3 \end{pmatrix} \right \rangle.
                \end{equation*}

%
\begin{code}
\begin{verbatim}
L<z3>:=CyclotomicField(3);
R<x,y,z> := PolynomialRing(L,3);
X:=Matrix(R,3,1,[x,y,z]);

H := MatrixGroup< 3, L| [1,0,0, 0,z3,0, 0,0,z3^2],[0,1,0, 0,0,1, 1,0,0],[1/(z3-z3^2),1/(z3-z3^2),1/(z3-z3^2), 1/(z3-z3^2),z3/(z3-z3^2),z3^2/(z3-z3^2), 1/(z3-z3^2),z3^2/(z3-z3^2),z3/(z3-z3^2)] >;


F:=Matrix(R,3,1,[x*y^3 - x*z^3,-x^3*y + y*z^3,x^3*z - y^3*z]);

F:=Matrix(R,3,1,[3*x^10*y^3 - 3*x^7*y^6 - 6*x^4*y^9 + 18*x^8*y^4*z - 36*x^5*y^7*z - 36*x^9*y^2*z^2 + 18*x^6*y^5*z^2 + 54*x^3*y^8*z^2 + 3*x^10*z^3 + 24*x^7*y^3*z^3 - 21*x^4*y^6*z^3 - 6*x*y^9*z^3 + 18*x^8*y*z^4 + 18*x^2*y^7*z^4 + 18*x^6*y^2*z^5 - 72*x^3*y^5*z^5 - 18*y^8*z^5 - 3*x^7*z^6 - 21*x^4*y^3*z^6 + 42*x*y^6*z^6 - 36*x^5*y*z^7 + 18*x^2*y^4*z^7 + 54*x^3*y^2*z^8 - 18*y^5*z^8 - 6*x^4*z^9 - 6*x*y^3*z^9, -6*x^9*y^4 - 3*x^6*y^7 + 3*x^3*y^10 - 36*x^7*y^5*z + 18*x^4*y^8*z + 54*x^8*y^3*z^2 + 18*x^5*y^6*z^2 - 36*x^2*y^9*z^2 - 6*x^9*y*z^3 - 21*x^6*y^4*z^3 + 24*x^3*y^7*z^3 + 3*y^10*z^3 + 18*x^7*y^2*z^4 + 18*x*y^8*z^4 - 18*x^8*z^5 - 72*x^5*y^3*z^5 + 18*x^2*y^6*z^5 + 42*x^6*y*z^6 - 21*x^3*y^4*z^6 - 3*y^7*z^6 + 18*x^4*y^2*z^7 - 36*x*y^5*z^7 - 18*x^5*z^8 + 54*x^2*y^3*z^8 - 6*x^3*y*z^9 - 6*y^4*z^9, -18*x^8*y^5 - 18*x^5*y^8 - 6*x^9*y^3*z + 42*x^6*y^6*z - 6*x^3*y^9*z + 18*x^7*y^4*z^2 + 18*x^4*y^7*z^2 + 54*x^8*y^2*z^3 - 72*x^5*y^5*z^3 + 54*x^2*y^8*z^3 - 6*x^9*z^4 - 21*x^6*y^3*z^4 - 21*x^3*y^6*z^4 - 6*y^9*z^4 - 36*x^7*y*z^5 - 36*x*y^7*z^5 + 18*x^5*y^2*z^6 + 18*x^2*y^5*z^6 - 3*x^6*z^7 + 24*x^3*y^3*z^7 - 3*y^6*z^7 + 18*x^4*y*z^8 + 18*x*y^4*z^8 - 36*x^2*y^2*z^9 + 3*x^3*z^10 + 3*y^3*z^10]);
GCD([F[1,1],F[2,1],F[3,1]]);

t:=Matrix(R,H.1)*X;
G:=Matrix(R,3,1,[Evaluate(F[1,1],[t[1,1],t[2,1],t[3,1]]),Evaluate(F[2,1],[t[1,1],t[2,1],t[3,1]]),Evaluate(F[3,1],[t[1,1],t[2,1],t[3,1]])]);
C:=Matrix(R,H.1)^(-1)*G;
C eq F;

t:=Matrix(R,H.2)*X;
G:=Matrix(R,3,1,[Evaluate(F[1,1],[t[1,1],t[2,1],t[3,1]]),Evaluate(F[2,1],[t[1,1],t[2,1],t[3,1]]),Evaluate(F[3,1],[t[1,1],t[2,1],t[3,1]])]);
C:=Matrix(R,H.2)^(-1)*G;
C eq F;

t:=Matrix(R,H.3)*X;
G:=Matrix(R,3,1,[Evaluate(F[1,1],[t[1,1],t[2,1],t[3,1]]),Evaluate(F[2,1],[t[1,1],t[2,1],t[3,1]]),Evaluate(F[3,1],[t[1,1],t[2,1],t[3,1]])]);
C:=Matrix(R,H.3)^(-1)*G;
C eq F;

a := Matrix(L,3,[1,0,0, 0,z3,0, 0,0,z3^2]);
phiH:=Inverse(NumberingMap(H));
MR:=Parent(a);
LH:=[];
for i:=1 to #H do
  z:=MR!(phiH(i));
  if z[1,1] ne 0 then
    y:=1/phiH(i)[1,1]*z;
  elif z[1,2] ne 0 then
    y:=1/phiH(i)[1,2]*z;
  else
    y:=1/phiH(i)[1,3]*z;
  end if;
  if not y in LH then
    Append(~LH,y);
  end if;
end for;
#LH;
S:=CyclicSubgroups(H);
for T in S do
  G:=T`subgroup;
  phiG:=Inverse(NumberingMap(G));
  LG:=[];
  for i:=1 to #G do
    z:=MR!(phiG(i));
    if z[1,1] ne 0 then
      y:=1/phiG(i)[1,1]*z;
    elif z[1,2] ne 0 then
      y:=1/phiG(i)[1,2]*z;
    else
      y:=1/phiG(i)[1,3]*z;
    end if;
    if not y in LG then
      Append(~LG,y);
    end if;
  end for;
print #LH,#LG;
end for;
\end{verbatim}
\end{code}

            \item[(F)] The group of order 72 generated by
                \begin{equation*}
                    \left \langle \begin{pmatrix} 1&0&0\\0&\zeta_3&0\\0&0&\zeta_3^2 \end{pmatrix},
                    \begin{pmatrix} 0&1&0\\0&0&1\\1&0&0 \end{pmatrix},
                    \frac{1}{\zeta_3-\zeta_3^2}\begin{pmatrix} 1&1&1\\1&\zeta_3&\zeta_3^2\\1&\zeta_3^2&\zeta_3 \end{pmatrix}, 
                    \frac{1}{\zeta_3-\zeta_3^2}\begin{pmatrix} 1&1&\zeta_3^2\\1&\zeta_3&\zeta_3\\ \zeta_3&1&\zeta_3 \end{pmatrix} \right \rangle. 
                \end{equation*}
%
%
\begin{code}
\begin{verbatim}
L<z3>:=CyclotomicField(3);
R<x,y,z> := PolynomialRing(L,3);
X:=Matrix(R,3,1,[x,y,z]);

H := MatrixGroup< 3, L| [1,0,0, 0,z3,0, 0,0,z3^2],[0,1,0, 0,0,1, 1,0,0],[1/(z3-z3^2),1/(z3-z3^2),1/(z3-z3^2), 1/(z3-z3^2),z3/(z3-z3^2),z3^2/(z3-z3^2), 1/(z3-z3^2),z3^2/(z3-z3^2),z3/(z3-z3^2)], [1/(z3-z3^2),1/(z3-z3^2),z3^2/(z3-z3^2), 1/(z3-z3^2),z3/(z3-z3^2),z3/(z3-z3^2), z3/(z3-z3^2),1/(z3-z3^2),z3/(z3-z3^2)] >;


a := Matrix(L,3,[1,0,0, 0,z3,0, 0,0,z3^2]);
phiH:=Inverse(NumberingMap(H));
MR:=Parent(a);
LH:=[];
for i:=1 to #H do
  z:=MR!(phiH(i));
  if z[1,1] ne 0 then
    y:=1/phiH(i)[1,1]*z;
  elif z[1,2] ne 0 then
    y:=1/phiH(i)[1,2]*z;
  else
    y:=1/phiH(i)[1,3]*z;
  end if;
  if not y in LH then
    Append(~LH,y);
  end if;
end for;
#LH;
S:=CyclicSubgroups(H);
for T in S do
  G:=T`subgroup;
  phiG:=Inverse(NumberingMap(G));
  LG:=[];
  for i:=1 to #G do
    z:=MR!(phiG(i));
    if z[1,1] ne 0 then
      y:=1/phiG(i)[1,1]*z;
    elif z[1,2] ne 0 then
      y:=1/phiG(i)[1,2]*z;
    else
      y:=1/phiG(i)[1,3]*z;
    end if;
    if not y in LG then
      Append(~LG,y);
    end if;
  end for;
print #LH,#LG;
end for;
\end{verbatim}
\end{code}

            \item[(G)] The group of order 216 generated by
                \begin{equation*}
                    \left \langle \begin{pmatrix} 1&0&0\\0&\zeta_3&0\\0&0&\zeta_3^2 \end{pmatrix},
                    \begin{pmatrix} 0&1&0\\0&0&1\\1&0&0 \end{pmatrix},
                    \frac{1}{\zeta3-\zeta_3^2}\begin{pmatrix} 1&1&1\\1&\zeta_3&\zeta_3^2\\1&\zeta_3^2&\zeta_3 \end{pmatrix},
                    \begin{pmatrix} \eta&0&0\\0&\eta&0\\0&0&\eta \zeta_3 \end{pmatrix} \right \rangle,
                \end{equation*}
                where $\eta^3 = \zeta_3^2$.

%
\begin{code}
\begin{verbatim}
K<z3>:=CyclotomicField(3);
R<x>:=PolynomialRing(K);
f:=x^3-z3^2;
L<e>:=ext<K|f>;
R<x,y,z> := PolynomialRing(L,3);
X:=Matrix(R,3,1,[x,y,z]);

H := MatrixGroup< 3, L| [1,0,0, 0,z3,0, 0,0,z3^2],[0,1,0, 0,0,1, 1,0,0],[1/(z3-z3^2),1/(z3-z3^2),1/(z3-z3^2), 1/(z3-z3^2),z3/(z3-z3^2),z3^2/(z3-z3^2), 1/(z3-z3^2),z3^2/(z3-z3^2),z3/(z3-z3^2)], [e,0,0, 0,e,0, 0,0,e*z3] >;

a := Matrix(L,3,[e,0,0, 0,z3,0, 0,0,z3^2]);
phiH:=Inverse(NumberingMap(H));
MR:=Parent(a);
LH:=[];
for i:=1 to #H do
  z:=MR!(phiH(i));
  if z[1,1] ne 0 then
    y:=1/phiH(i)[1,1]*z;
  elif z[1,2] ne 0 then
    y:=1/phiH(i)[1,2]*z;
  else
    y:=1/phiH(i)[1,3]*z;
  end if;
  if not y in LH then
    Append(~LH,y);
  end if;
end for;
#LH;
S:=CyclicSubgroups(H);
for T in S do
  G:=T`subgroup;
  phiG:=Inverse(NumberingMap(G));
  LG:=[];
  for i:=1 to #G do
    z:=MR!(phiG(i));
    if z[1,1] ne 0 then
      y:=1/phiG(i)[1,1]*z;
    elif z[1,2] ne 0 then
      y:=1/phiG(i)[1,2]*z;
    else
      y:=1/phiG(i)[1,3]*z;
    end if;
    if not y in LG then
      Append(~LG,y);
    end if;
  end for;
print #LH,#LG;
end for;
\end{verbatim}
\end{code}

        \end{enumerate}
    \subsubsection{Sylow Subgroups}
        \begin{enumerate}
            \item[(H)] The group of order 60 generated by
                \[ \left \langle \begin{pmatrix} 0&1&0\\0&0&1\\1&0&0 \end{pmatrix},
                   \begin{pmatrix} 1&0&0\\0&-1&0\\0&0&-1 \end{pmatrix},
                   \begin{pmatrix} -1&\mu_2&\mu_1\\ \mu_2&\mu_1&-1\\ \mu_1&-1&\mu_2 \end{pmatrix} \right \rangle,\]
                where $\mu_1=\frac{1}{2}(-1+\sqrt{5})$ and $\mu_2 = \frac{1}{2}(-1-\sqrt{5})$.


\begin{code}
\begin{verbatim}
R<x>:=PolynomialRing(Rationals());
f:=x^2-5;
L<a>:=NumberField(f);
mu1:=1/2*(-1+a);
mu2:=1/2*(-1-a);
R<x,y,z> := PolynomialRing(L,3);
X:=Matrix(R,3,1,[x,y,z]);

H := MatrixGroup< 3, L| [0,1,0, 0,0,1, 1,0,0], [1,0,0, 0,-1,0, 0,0,-1], [-1/2,mu2/2,mu1/2, mu2/2,mu1/2,-1/2, mu1/2,-1/2,mu2/2] >;

F:=Matrix(R,3,1,[35/4*x^5 + 1/4*(-5*a + 65)*x^3*y^2 + 1/4*(5*a + 65)*x^3*z^2 + 1/8*(5*a + 65)*x*y^4 + 45/2*x*y^2*z^2 + 1/8*(-5*a + 65)*x*z^4,1/8*(-5*a + 65)*x^4*y + 1/4*(5*a + 65)*x^2*y^3 + 45/2*x^2*y*z^2 + 35/4*y^5 + 1/4*(-5*a + 65)*y^3*z^2 + 1/8*(5*a + 65)*y*z^4,1/8*(5*a + 65)*x^4*z + 45/2*x^2*y^2*z + 1/4*(-5*a + 65)*x^2*z^3 + 1/8*(-5*a+ 65)*y^4*z + 1/4*(5*a + 65)*y^2*z^3 + 35/4*z^5]);
GCD([F[1,1],F[2,1],F[3,1]]);

b := Matrix(L,3,[a,0,0, 0,1,0, 0,0,1]);
phiH:=Inverse(NumberingMap(H));
MR:=Parent(b);
LH:=[];
for i:=1 to #H do
  z:=MR!(phiH(i));
  if z[1,1] ne 0 then
    y:=1/phiH(i)[1,1]*z;
  elif z[1,2] ne 0 then
    y:=1/phiH(i)[1,2]*z;
  else
    y:=1/phiH(i)[1,3]*z;
  end if;
  if not y in LH then
    Append(~LH,y);
  end if;
end for;
#LH;
S:=CyclicSubgroups(H);
for T in S do
  G:=T`subgroup;
  phiG:=Inverse(NumberingMap(G));
  LG:=[];
  for i:=1 to #G do
    z:=MR!(phiG(i));
    if z[1,1] ne 0 then
      y:=1/phiG(i)[1,1]*z;
    elif z[1,2] ne 0 then
      y:=1/phiG(i)[1,2]*z;
    else
      y:=1/phiG(i)[1,3]*z;
    end if;
    if not y in LG then
      Append(~LG,y);
    end if;
  end for;
print #LH,#LG;
end for;
\end{verbatim}
\end{code}
%

            \item[(I)] The group of order 360 generated by
                \begin{equation*}
                    \left \langle \begin{pmatrix} 0&1&0\\0&0&1\\1&0&0 \end{pmatrix},
                   \begin{pmatrix} 1&0&0\\0&-1&0\\0&0&-1 \end{pmatrix},
                   \begin{pmatrix} -1&\mu_2&\mu_1\\ \mu_2&\mu_1&-1\\ \mu_1&-1&\mu_2 \end{pmatrix}, \begin{pmatrix} -1 &0&0 \\ 0&0&-\zeta_3  \\ 0&-\zeta_3^2&0 \end{pmatrix} \right \rangle,
                \end{equation*}
                where $\mu_1=\frac{1}{2}(-1+\sqrt{5})$ and $\mu_2 = \frac{1}{2}(-1-\sqrt{5})$.


\begin{code}
\begin{verbatim}
K<z3>:=CyclotomicField(3);
R<x>:=PolynomialRing(K);
f:=x^2-5;
L<a>:=ext<K|f>;
mu1:=1/2*(-1+a);
mu2:=1/2*(-1-a);
R<x,y,z> := PolynomialRing(L,3);
X:=Matrix(R,3,1,[x,y,z]);

H := MatrixGroup< 3, L| [0,1,0, 0,0,1, 1,0,0], [1,0,0, 0,-1,0, 0,0,-1], [-1/2,mu2/2,mu1/2, mu2/2,mu1/2,-1/2, mu1/2,-1/2,mu2/2], [-1,0,0, 0,0,-z3, 0,-z3^2,0] >;

T:=InvariantRing(H);
A:=InvariantsOfDegree(T,6)[1];
B:=InvariantsOfDegree(T,12)[1];

AssignNames(~T,["x","y","z"]);

Ax:=Derivative(A,1);
Ay:=Derivative(A,2);
Az:=Derivative(A,3);

Bx:=Derivative(B,1);
By:=Derivative(B,2);
Bz:=Derivative(B,3);


F:=Matrix(R,3,1,[Ay*Bz-Az*By,Az*Bx-Ax*Bz,Ax*By-Ay*Bx]);
GCD([F[1,1],F[2,1],F[3,1]]);

b := Matrix(L,3,[a,0,0, 0,1,0, 0,0,1]);
phiH:=Inverse(NumberingMap(H));
MR:=Parent(b);
LH:=[];
for i:=1 to #H do
  z:=MR!(phiH(i));
  if z[1,1] ne 0 then
    y:=1/phiH(i)[1,1]*z;
  elif z[1,2] ne 0 then
    y:=1/phiH(i)[1,2]*z;
  else
    y:=1/phiH(i)[1,3]*z;
  end if;
  if not y in LH then
    Append(~LH,y);
  end if;
end for;
#LH;
S:=CyclicSubgroups(H);
for T in S do
  G:=T`subgroup;
  phiG:=Inverse(NumberingMap(G));
  LG:=[];
  for i:=1 to #G do
    z:=MR!(phiG(i));
    if z[1,1] ne 0 then
      y:=1/phiG(i)[1,1]*z;
    elif z[1,2] ne 0 then
      y:=1/phiG(i)[1,2]*z;
    else
      y:=1/phiG(i)[1,3]*z;
    end if;
    if not y in LG then
      Append(~LG,y);
    end if;
  end for;
print #LH,#LG;
end for;
\end{verbatim}
\end{code}


            \item[(J)] The group of order 168 generated by
                \begin{equation*}
                    \left \langle \begin{pmatrix} 1&0&0\\0&\zeta_7&0\\ 0&0&\zeta_7^3\end{pmatrix},
                    \begin{pmatrix} 0&1&0\\0&0&1\\1&0&0 \end{pmatrix},
                    \begin{pmatrix} \zeta_7^4-\zeta_7^3&\zeta_7^2-\zeta_7^5&\zeta_7-\zeta_7^6\\
                    \zeta_7^2-\zeta_7^5&\zeta_7-\zeta_7^6&\zeta_7^4-\zeta_7^3\\
                    \zeta_7-\zeta_7^6&\zeta_7^4-\zeta_7^3&\zeta_7^2-\zeta_7^5 \end{pmatrix} \right \rangle.
                \end{equation*}


\begin{code}
\begin{verbatim}
L<z7>:=CyclotomicField(7);
h:=(z7+z7^2+z7^4-z7^6-z7^5-z7^3);
a:=z7^4-z7^3;
b:=z7^2-z7^5;
c:=z7-z7^6;
H := MatrixGroup< 3, L| [z7,0,0, 0,z7^2,0, 0,0,z7^4], [0,1,0, 0,0,1, 1,0,0], [a/h,b/h,c/h, b/h,c/h,a/h, c/h,a/h,b/h] >;

F:=7/2*Matrix(R,3,1,[94/7*x^7*z^2 - 436*x^5*y^3*z + 374*x^4*y*z^4 + 52*x^3*y^6 + 374*x^2*y^4*z^3 - 218*x*y^2*z^6 + 810/7*y^7*z^2 + 90/7*z^9,90/7*x^9 - 218*x^6*y*z^2 + 374*x^4*y^4*z + 374*x^3*y^2*z^4 + 94/7*x^2*y^7 + 810/7*x^2*z^7 - 436*x*y^5*z^3 + 52*y^3*z^6, 810/7*x^7*y^2 + 52*x^6*z^3 + 374*x^4*y^3*z^2 - 436*x^3*y*z^5 - 218*x^2*y^6*z + 374*x*y^4*z^4 + 90/7*y^9 + 94/7*y^2*z^7]);

b := Matrix(L,3,[z7,0,0, 0,1,0, 0,0,1]);
phiH:=Inverse(NumberingMap(H));
MR:=Parent(b);
LH:=[];
for i:=1 to #H do
  z:=MR!(phiH(i));
  if z[1,1] ne 0 then
    y:=1/phiH(i)[1,1]*z;
  elif z[1,2] ne 0 then
    y:=1/phiH(i)[1,2]*z;
  else
    y:=1/phiH(i)[1,3]*z;
  end if;
  if not y in LH then
    Append(~LH,y);
  end if;
end for;
#LH;
S:=CyclicSubgroups(H);
for T in S do
  G:=T`subgroup;
  phiG:=Inverse(NumberingMap(G));
  LG:=[];
  for i:=1 to #G do
    z:=MR!(phiG(i));
    if z[1,1] ne 0 then
      y:=1/phiG(i)[1,1]*z;
    elif z[1,2] ne 0 then
      y:=1/phiG(i)[1,2]*z;
    else
      y:=1/phiG(i)[1,3]*z;
    end if;
    if not y in LG then
      Append(~LG,y);
    end if;
  end for;
print #LH,#LG;
end for;
\end{verbatim}
\end{code}
        \end{enumerate}

\section{Bound of the order of the automorphism group in terms of degree for $\P^2$} \label{sect_bound}
    Since elements of the automorphism group must permute the points of period $n$ for each $n$, it is not hard to get an explicit (factorial) upper bound on the size of the automorphism group depending on the existence of periodic points. Levy has proven the stronger result of the existence of a bound on the size of the automorphism group depending on the dimension and degree of the map \cite{Levy}. In particular he shows that the largest abelian subgroup is bounded by $d^{N+1}$ and then uses the result of G.A.Miller which states that the size of a finite group is bounded in terms of its largest abelian subgroup.
    However, his proof does not result in an explicit constant. Levy's ideas can easily be used to deduce a quite strong bound in dimension one using the explicit description of the finite subgroups of $\PGL_2$ given by Silverman \cite{Silverman12}. In particular, the only two subgroups with arbitrarily larger order are the cyclic group and the dihedral group. Since it is possible to bound the order of the cyclic group in terms of the degree \cite{Silverman12} and the largest cyclic subgroup of the dihedral group has index two, it is not hard to produce the explicit bound of $\#\Aut \leq \max(60,2(d+1))$. In this section, we use the explicit description of the finite subgroups of $\PGL_3$ to produce an explicit bound on the size of the automorphism group depending on the degree.

    We first bound the index of the largest cyclic subgroup for the finite subgroups of $\PGL_3$.

    \begin{lem} \label{lem_index}
        We have the following upper bound on the index $m$ of the largest cyclic subgroup for the finite subgroups of $\PGL_3$ given in Section \ref{sect.P2}. We denote by $n$ the order of the cyclic subgroup.
        \begin{enumerate}
            \item[(A)] $m = 1$
            \item[(B)]
                \begin{enumerate}
                    \item[(B1)] $m \leq 2n$
                    \item[(B2)] $m \leq 12$
                    \item[(B3)] $m \leq 24$
                    \item[(B4)] $m \leq 60$
                \end{enumerate}
            \item[(C)] $m \leq 3n$.

            \item[(D)] $m \leq 6n$
            \item[(E)] $m \leq 9$
            \item[(F)] $m \leq 18$
            \item[(G)] $m \leq 36$
            \item[(H)] $m \leq 12$
            \item[(I)] $m \leq 72$
            \item[(J)] $m \leq 24$
        \end{enumerate}
        In particular, the index of the largest cyclic subgroup is at most $6n$.
    \end{lem}
    Note that these bounds are not necessarily optimal.

    \begin{proof}
        We compute the index using Lagrange's Theorem for finite subgroups.
        \begin{enumerate}
        \item[(A)] The whole group is cyclic so that $m=1$.
        \item[(B)] For all four of these groups, the group elements are of the form
            \begin{equation*}
                \begin{pmatrix}
                  \zeta_p &0&0\\
                  0 &a&b\\
                  0&c&d
                \end{pmatrix}
             \end{equation*}
             where the lower right 2x2 is one of the $\PGL_2$ subgroups and $\zeta_p$ is a primitive $p$-th root of unity.
            \begin{enumerate}
                \item[(B1)] If the dihedral part is order $2n$ and the upper left entry is a $p$-th root of unity, then the total order of the group is $2np$.
                    The dihedral group has a cyclic subgroup of order $n$. If $p\mid n$, then the total group has a cyclic subgroup of at most order $n$. Since in this case $p<n$ the total order is $\leq 2n^2$, the index is at most $\frac{2n^2}{n}=2n$. If $p > n$, then the cyclic group is order at least $p$ and the index is at most $\frac{2np}{p} = 2n$.

                \item[(B2)] We can have at most a total of $12p$ elements and at least a cyclic group of order $p$, so the index is at most $\frac{12p}{p} = 12$.
                \item[(B3)] We can have at most a total of $24p$ elements and at least a cyclic group of order $p$, so the index is at most $\frac{24p}{p} = 24$.
                \item[(B4)] We can have at most a total of $60p$ elements and at least a cyclic group of order $p$, so the index is at most $\frac{60p}{p} = 60$.
            \end{enumerate}

        \item[(C)] Let $A$ denote the generator of the cyclic part and $B$ the permutation. Then we have the relations
            \begin{align*}
                A^n &= id\\
                B^3 &= id\\
                BA &= ABAB^{-1}A^{-1}B.
            \end{align*}
            The last relation says we can always move the permutation to the right-hand side of the cyclic matrix. That means we have three non-zero entries which each can be one of the $n$ possible roots of unity and can be permuted in three different ways. Since we can always normalize one of those entries to be $1$, the largest possible order of the group is $3n^2$.
            There is a cyclic subgroup of order $n$, so we have the index of the largest cyclic subgroup to be at most $\frac{3n^2}{n} = 3n$.

        \item[(D)] Label the generators as $A,B,C$ in the order listed above. We produce an upper bound on the order of the group as follows. If we take the entries of $C$ to be $1$, then $B,C$ together generate the full permutation group on $3$ elements (order $6$). Then each element contains three non-zero entries that can be any of the $n$-th roots of unity and we can again normalize for a total of $n^2$ possibilities. These can then be permuted in any way for a total of $6n^2$ entries.
            There is a cyclic group of order $n$ so the index of the largest cyclic subgroup is at most $\frac{6n^2}{n} = 6n$.

        \item[(E)] There is a cyclic subgroup of order 4 which has index $\frac{36}{4}=9$.
        \item[(F)] There is a cyclic subgroup of order 4 which has index $\frac{72}{4}=18$.
        \item[(G)] There is a cyclic subgroup of order 6 which has index $\frac{216}{6}=36$.
        \item[(H)] There is a cyclic subgroup of order 5 which has index $\frac{60}{5}=12$.
        \item[(I)] There is a cyclic subgroup of order 5 which has index $\frac{360}{5}=72$.
        \item[(J)] There is a cyclic subgroup of order 7 which has index $\frac{168}{7}=24$.
    \end{enumerate}
    \end{proof}

    \begin{thm} \label{thm_degree_bound}
        Let $f:\P^2 \to \P^2$ be a morphism of degree $d \geq 2$. Then,
        \begin{equation*}
            \#\Aut(f) \leq 6d^6.
        \end{equation*}
    \end{thm}
    \begin{proof}
        By Levy \cite{Levy}, if $C_n \in \Aut(f)$, then 
        \begin{equation*}
            d^3 \geq n.
        \end{equation*}
        We also have by Lagrange's Theorem for finite subgroups that
        \begin{equation*}
            n = \frac{1}{[C_n:\Aut(f)]}\#\Aut(f).
        \end{equation*}
        Now applying Lemma \ref{lem_index}, we have
        \begin{equation*}
            n \geq \frac{\#\Aut(f)}{6n},
        \end{equation*}
        So that
        \begin{equation*}
            6n^2 \geq \#\Aut(f).
        \end{equation*}
        Finally, recalling that $n \leq d^3$ we have
        \begin{equation*}
            \#\Aut(f) \leq 6d^6.
        \end{equation*}
    \end{proof}
    This bound is unlikely to be optimal. The largest example of a cyclic subgroup (relative to degree) known to the authors is the following example\footnote{Thanks to Joseph Silverman for making the authors aware of this example for $d=2$}.
    \begin{exmp}
        Maps of the form
        \begin{align*}
            f:\P^2 &\to \P^2\\
            (x,y,z) &\mapsto (y^d,z^d,x^d)
        \end{align*}
        have a cyclic automorphism of order $n=d^2+d+1$. Note the exponent $2$ here instead of the exponent $3$ used from Levy's result in the proof of Theorem \ref{thm_degree_bound}. Essentially, we need to solve a system of linear congruences. Dehomogenizing $f$, we are working with the map
        \begin{equation*}
            F(x,y) = \left(\frac{y^d}{x^d},\frac{1}{x^d}\right).
        \end{equation*}
        If we are conjugating by $\alpha:(x,y) \to (\zeta_n^ax,\zeta_n^by)$ for some integers $a$ and $b$ and primitive $n$-th root of unity $\zeta_n$, to have $\alpha \in \Aut(f)$ we need to solve the congruences
        \begin{align*}
            db-(d+1)a \equiv 0 \pmod{n}\\
            -da-b \equiv 0 \pmod{n}.
        \end{align*}
        This results in
        \begin{equation*}
            d^2+d+1 \equiv 0 \pmod{n}.
        \end{equation*}
    \end{exmp}

\section{Algorithms for determining automorphism groups on $\P^N$} \label{sect.alg}
    Faber-Manes-Viray \cite{FMV} give several algorithms for $\P^1$ to determine the automorphism group of a given map. The goal of this section is to extend their algorithms to $\P^2$.

    \subsection{The Automorphism Scheme}
    We first make the obvious generalizations from $\P^1$ to $\P^N$ of the automorphism scheme and conjugation scheme defined in Faber-Manes-Viray \cite{FMV}. The method of proof remains the same with the details modified for dimension $N>1$.

    We need one helper lemma from Petsche-Szpiro-Tepper \cite{petsche}.
    \begin{defn}
        Let $f:\P^N \to \P^N$ be a morphisms of degree at least $1$. We say $f$ has good reduction over $K$ if there exists a choice of coordinates $\textbf{x}$ on $\P^N_K$ such that $f$ extends to an endomorphism of the associated integral model $\P^N_{K^{\circ}}$.

        Let $f:K^{N+1} \to K^{N+1}$ be a homogeneous map of degree $d \geq 1$. We say that $f$ has \emph{nonsingular reduction} over $K$ if $f$ is defined over $K^{\circ}$ and the reduced map over the residue field is nonsingular.
    \end{defn}

    \begin{lem}\cite[Lemma 6]{petsche}\label{lem_reduction}
        Let $K$ be a field with a nontrivial, nonarchimedean absolute value. Let $f:\P^N_K \to \P^N_k$ be a morphism of degree at least $2$. Let $f(x),g(y)$ be models of $f$ with respect to the coordinates $x,y$ where $\alpha(x) = y$ and $f(x) = \alpha^{-1} \circ g \circ \alpha$ for some $\alpha \in \GL_{N+1}(K)$. If both $f,g$ have nonsingular reduction, then $\alpha \in \GL_{N+1}(K^{\circ})$.
    \end{lem}

    Let $R$ be a noetherian commutative ring with identity, and let $\Ralg$ and $\Grp$ denote the categories of commutative $R$-algebras and (abstract) groups,  respectively. For any $R$-algebra $S$, we identify $\PGL_{N+1}(S)$ with $\Aut(\P^N_S)$, the group of automorphisms of $\P^N$ defined over $S$.
    \begin{defn}
        Let $f:\P^N \to \P^N$ be a morphism of degree at least $2$. Let $\Aut_f$ denote the functor from $\Ralg$ to $\Grp$ given by
        \begin{equation*}
            S \mapsto \Aut(f)(S) = \{ \alpha \in \PGL_{N+1}(S) \col f=f^{\alpha}\}.
        \end{equation*}
        The functor acts on $R$-algebra morphisms by base extension of the associated group of automorphisms.
    \end{defn}

    \begin{thm}\label{thm_aut_scheme}
        Let $R$ be a noetherian commutative ring and let $f:\P^N_R \to \P^N_R$ be a morphism of degree at least $2$. Then the functor $\Aut_f$ is represented by a closed finite $R$-subgroup scheme $\Aut(f) \subset \PGL_{N+1}$.
    \end{thm}
    \begin{rem}
        It is not flat \cite[Remark 2.2]{FMV}.
    \end{rem}
    \begin{proof}
        Fix a noetherian commutative ring $R$. Over $R$, $\PGL_{N+1}$ may be embedded as an affine subvariety of $\P^{N^2-1}_R = \Proj R[\alpha_{ij}]$, where $\alpha = (\alpha_{ij}) \in \PGL_{N+1}$. In particular, it is the complement of the determinant of $(\alpha_{ij})$. Let $f:\P^N_R \to \P^N_R$ be a nonconstant morphism. We may define $\Aut_f$ as subgroup scheme of $\PGL_{N+1}$ explicitly as follows. After fixing coordinates of $\P^{N}_R$ as $\bar{x}$. Let $f=[f_0,\ldots,f_N]$, where $f_i$ are homogeneous polynomials in the variables $\bar{x}$ of degree $d$ with coefficients in $R$ such that their Macaulay resultant is a unit in $R$. Similarly, for any $R$-algebra $S$, given an $\alpha \in \PGL_{N+1}(S)$ we may write
        \begin{equation*}
            f^{\alpha} = \alpha \circ f \circ \alpha^{-1}: \P^N_S \to \P^N_S
        \end{equation*}
        as homogeneous polynomials of degree $d$ in the variables $\bar{x}$ whose coefficients are homogeneous polynomials $[g_0,\ldots, g_N]$ in $R[\alpha_{ij}]$. For $\alpha \in \Aut(f)$, we must have
        \begin{equation*}
            f^{\alpha} = f,
        \end{equation*}
        which are homogeneous equations
        \begin{equation*}
            f_ig_j - f_jg_i = 0 \quad i,j \in \{0,\ldots,N\}.
        \end{equation*}
        Each such equation gives $\binom{d+N}{N}$ equations (one for each coefficient), and they define a closed subscheme of $\PGL_{N+1}$ defined over $R$. It is easy to see that $\Aut_f(S)$ is a subgroup of $\PGL_{N+1}(S)$ under composition for every $S$.

        Now we show finiteness when $d \geq 2$. We claim the map $\Aut_f \to \Spec{R}$ is quasi-finite. It suffices to check this for geometric fibers, and it is known that $\Aut_f(K)$ is a finite group for any algebraically closed field $K$ \cite[Prop 4.65]{ADS}. To see $\Aut_f$ is proper over $\Spec{R}$, we use the valuative criterion of properness. Let $o$ be a discrete valuation ring with field of fractions $k$ and consider the following diagram
        \begin{equation*}
            \xymatrix{
            \Spec{k} \ar[r] \ar[d] & \Aut_f \ar[d]\\
            \Spec{o} \ar[r] \ar@{-->}^{f}[ur] & \Spec{R}}.
        \end{equation*}
        The left vertical map is the canonical open immersion and the right vertical map is the structure morphism. We must show there is a unique $\phi:\Spec{o} \to \Aut_f$, making the diagram commute. Without loss of generality we may assume that $R = o$ and the bottom map is the identity. Since $f$ is defined over $o$ it has good reduction. We give $k$ the structure of a nonarchimedean field by defining for $x \in k$, $\abs{x} = e^{-v(x)}$, where $v$ is the canonical extension of the absolute value on $o$. Lemma \ref{lem_reduction} asserts that every $k$-automorphism of $f$ also has good reduction. Equivalently, every $k$-valued point may be extended to an $o$-valued point.

        We have shown that $\Aut_f \to \Spec{R}$ is a quasi-finite proper morphism. Zariski's main theorem tells us that it factors as an open immersion of $R$-schemes $\Aut_f \to X$ followed by a finite morphism $X \to \Spec{R}$. But $\Aut_f$ is proper, so any open immersion is actually an isomorphism. Hence, $\Aut_f$ is finite over $\Spec{R}$.
    \end{proof}

    \begin{defn}
        Let $f,g:\P^N_R \to \P^N_R$ be two morphisms of degree $d \geq 2$. Write $\Set$ for the category of sets. Let $\Conj_{f,g}:\Ralg \to \Set$ denote the functor given by
        \begin{equation*}
            \Conj_{f,g}(S) = \{ \alpha \in \PGL_{N+1}(S) \col f^{\alpha} = g\}.
        \end{equation*}
        The functor $\Conj_{f,g}$ acts on $R$-Algebra morphisms by base extension of the associated conjugation maps.
    \end{defn}
    As with the automorphism functor, we can easily generalize the results of Faber-Manes-Viray \cite{FMV}.
    \begin{thm}
        Let $R$ be a neotherian commutative ring with identity. Let $f,g:\P^N_R \to \P^N_R$ be two morphisms of degree $d \geq 2$. Then $\Conj_{f,g}$ is represented by a closed finite $R$-subscheme $\Conj_{f,g} \subset \PGL_{N+1}$.
    \end{thm}
    \begin{proof}
        Similar to the proof for $\Aut_f$ except that we define the equations with
        \begin{equation*}
            f^{\alpha} = g.
        \end{equation*}

        To get properness we use a simple generalization of Lemma \ref{lem_reduction}.
    \end{proof}
    \begin{rem}
        The group scheme $\PGL_{N+1}$ has relative dimension $(N+1)^2-1$ over $R$, while $\Rat_d^N$ has relative dimension $\binom{d+N+1}{N+1}$ over $R$. For $d\geq 2$ and $N \geq 1$, we have $\binom{d+N+1}{N+1} > (N+1)^2-1$, so for a fixed $f$, a general choice of $g$ will result in $\Conj_{f,g} = \emptyset$.
    \end{rem}
    \begin{rem}
        When $\Conj_{f,g}$ is nonempty it is the principal homogeneous space for $\Aut_f$.
    \end{rem}

\subsection{Method of Invariant Sets}
    In this section, we generalize the method of invariants sets from Faber-Manes-Viray \cite{FMV}. The idea is to find two sets $T_f$ and $T_g$ containing $N+2$ independent points such that
    \begin{equation*}
        \alpha(T_f) = T_g \quad \text{ for all } \alpha \in \Conj_{f,g}.
    \end{equation*}
    Then it is a matter of linear algebra to find the possible $\alpha$ sending one set to the other since
    \begin{equation*}
        \Conj_{f,g} \subseteq \Hom(T_f,T_g).
    \end{equation*}
    Of course, the larger the dimension $N$, the larger the number of possible ways to map $T_f$ to $T_g$, so the combinatorics quickly can become unmanageable in practice.

    Note that we can assume $\deg(f) = \deg(g)$ since otherwise $\Conj_{f,g} = \emptyset$. Because conjugation preserves dynamics, i.e., points of period $n$ go to points of period $n$, one way to construct invariant sets is to consider the periodic points of $f$ and $g$. Similarly for pre-images: for a fixed point $P \in \P^N$, we have the sets $f^{-n}(P)$ are in bijective correspondence with the set $g^{-n}(\alpha(P))$ for each $n \geq 1$. The number of points of period $n$ for a morphisms of $\P^N$ is
    \begin{equation*}
        \sum_{i=0}^N d^{in}
    \end{equation*}
    when counted with multiplicity. This gives two possible approaches to finding invariant sets:
    \begin{enumerate}
        \item Find an integer $m$, such that the set of $m$ periodic points $f$ has at least $N+2$ independent points.
        \item Find an integer $m$, such that the set of $m$-th preimages of the fixed points of $f$ has at least $N+2$ independent points.
    \end{enumerate}
    \begin{rem}
        Note that that these sets must have the same cardinality for $f$ and $g$, otherwise $\Conj_{f,g} = \emptyset$.
    \end{rem}

    The number of candidate $\alpha$ is approximately $\#T_f^{N+2}/((N+2)!) = O(d^{N(N+2)}/((N+2)!))$. So for $d$ or $N$ large, this quickly becomes infeasible. Even for small $d,N$, its primary implementation drawback is the need to work with large degree number field extensions to find all of the elements of $T_f$. Eliminating this issue is the purpose of the fixed point method in the next section.

\subsection{Method of Fixed Points}
    Note that, in some sense, the method of fixed points is a specialization of the method of invariant sets for $\Conj_{f,g}$ using that fact that $f=g$.

    Let $f: \P^2_K \rightarrow \P^2_K$ be a morphism such that $\deg(f) > 1$ and $K$ is a number field ($\Char(K) =0$). A key fact used in the algorithm is that $f$ permutes the fixed points of its automorphisms. Since $f$ is an endomorphism of $\P^2$, its automorphisms will be elements of $\PGL_3$ given as $3\times 3$ matrices. Thus, the fixed points can be determined as the eigenvectors of the automorphism's eigenspace.

    It is possible for a $3 \times 3$ matrix to have 1, 2, or 3 independent eigenvalues corresponding to 1, 2, or 3 fixed points of the automorphism. In particular, the case when we have less than $3$ linearly independent eigenvectors, these matrices are conjugate to some matrix in Jordan canonical form. The Jordan forms that are not diagonalizable are

    \begin{equation*}
        \begin{pmatrix} \lambda&1&0 \\ 0&\lambda&1 \\ 0&0&\lambda \end{pmatrix} ,\,
        \begin{pmatrix} \lambda_1&1&0 \\ 0&\lambda_1&0 \\ 0&0&\lambda \end{pmatrix},
    \end{equation*}

where $\lambda_i$ are the eigenvalues of the automorphism. The following two lemmas show that matrices in these two forms cannot have finite order when $\Char(K) = 0$, which means that there can be no automorphisms conjugate to these forms (due to the finiteness of $\Aut(f)$).

    \begin{lem}
        Let $\lambda$ be a constant and $n$ a positive integer. Then we have
        \begin{equation*}
            \begin{pmatrix} \lambda&1&0 \\ 0&\lambda&1 \\ 0&0&\lambda \end{pmatrix}^n =
            \begin{pmatrix} \lambda^n &n\lambda^{n-1}&\frac{n(n-1)}{2}\lambda^{n-2} \\ 0&\lambda^ n& n\lambda^{n-1} \\ 0&0&\lambda^n \end{pmatrix}.
        \end{equation*}
    \end{lem}
    \begin{proof}
        Simple calculation
    \end{proof}
    \begin{lem}
        Let $\lambda,\lambda_1$ be distinct constants and $n$ a positive integer. Then we have
        \begin{equation*}
            \begin{pmatrix} \lambda_1&1&0 \\ 0&\lambda_1&0 \\ 0&0&\lambda \end{pmatrix}^n =
            \begin{pmatrix} \lambda_1^n &n\lambda_1^{n-1}&0 \\ 0&\lambda_1^n& 0 \\ 0&0&\lambda^n \end{pmatrix}.
        \end{equation*}
    \end{lem}
    \begin{proof}
        Simple calculation.
    \end{proof}

    The remaining case to consider is where the automorphism has three linearly independent fixed points. The following proposition describes the form such an autmorphism must have.

	\begin{prop} \label{prop:auto_form}
        Let $f: \P^2 \rightarrow \P^2$ and $S \in \Aut(f)$ of order $n$. Denote $x,y,z$ as the three distinct fixed points of $S$. If $x,y,z$ are not all colinear, then there exists $U \in \PGL_3$ such that $U(x) = (1,0,0)$, $U(y) = (0,1,0)$, and $U(z) = (0,0,1)$ and $S \in \Aut(f)$ is given by
		\begin{equation*}
            S = U^{-1} \begin{pmatrix} \zeta_n^a&0&0\\0&\zeta_n^b&0 \\0&0&1 \end{pmatrix} U
        \end{equation*}
        where $\zeta_n$ is a $n$-th root of unity and $a,b \in \Z$ such that $gcd(a,n) = gcd(b,n) = 1$.
	\end{prop}
	
	\begin{proof}
		If $x \in \Fix(S)$, then $f(x) \in \Fix(S)$:
		\begin{equation*}
            f(S(x)) = S(f(x)) = S(x),
        \end{equation*}
        where the first equality is true because $S \in \Aut(f)$. Now let $S = U^{-1} M U$, where $M \in \PGL_3$. Multiplying on the left by $U$ gives us
		$$US = MU$$.
        Since $x,y,z \in \Fix(S)$, we get the following three equations for some constants $a,b,c$
		\begin{align*}
            \begin{pmatrix}a\\ 0 \\0 \end{pmatrix} &= USx = MUx = M\begin{pmatrix}1\\ 0 \\0 \end{pmatrix}\\
            \begin{pmatrix}0\\ b \\0 \end{pmatrix} &= USy = MUy = M\begin{pmatrix}0\\ 1 \\0 \end{pmatrix}\\
            \begin{pmatrix}0\\ 0 \\c \end{pmatrix} &= USz = MUz = M\begin{pmatrix}0\\ 0 \\1 \end{pmatrix}.
        \end{align*}
        Normalizing the matrix $M$ we have
		\begin{align*}
            M = \begin{pmatrix} a/c&0&0\\0& b/c&0 \\0&0&1 \end{pmatrix}.
        \end{align*}
        Since $\Aut(f)$ is a finite group, there exists an $n$ such that $S^n(x) = S(x)$, so both $\frac{a}{c}$ and $\frac{b}{c}$ are $n$-th roots of unity.
	\end{proof}

    In the first part of the proof for Proposition \ref{prop:auto_form}, we show that if $x \in \Fix(S)$ then $f(x) \in \Fix(S)$. In other words, we show that $f$ permutes the fixed points of its automorphisms. Therefore, there are seven different possible actions of $f$ on the fixed points of its automorphisms. We label the actions as $s_1,\ldots,s_7$. In the following cases, $x,y,z$ represent the three fixed points of the automorphism.
    \begin{enumerate} \label{sec:comb}
    	\item[$(s_1)$]: $f$ fixes $x,y,z$
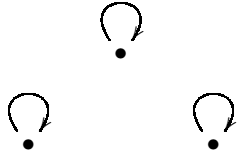
        	\begin{equation*}
        		\xymatrix{
        		&\bullet \ar@(ul,ur)& \\
        		\bullet \ar@(ul,ur) && \bullet \ar@(ul,ur)\\}
        	\end{equation*}

    	\item[$(s_2)$]: $f$ permutes $x,y,z$
        	
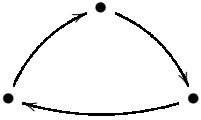
\begin{equation*}
        		\xymatrix{
        		&\bullet \ar@/^/ [rd]& \\
        		\bullet \ar@/^/[ru]&& \bullet \ar@/^/[ll]\\
        			}
        	\end{equation*}

        \item[$(s_3)$]: $f$ permutes $x,y$ with a fixed point $z$
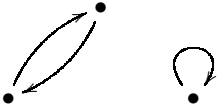
        	\begin{equation*}
        		\xymatrix{
        		&\bullet \ar@/^/ [ld]& \\
        		\bullet \ar@/^/[ru]&& \bullet \ar@(ul,ur)\\
        			}
        	\end{equation*}

        \item[$(s_4)$]: $f$ permutes $x,y$ with a preperiodic point $z$
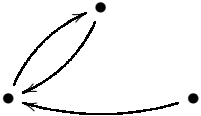
        	\begin{equation*}
        		\xymatrix{
        		&\bullet \ar@/^/ [ld]& \\
        		\bullet \ar@/^/[ru]&& \bullet \ar@/^/[ll]\\
        			}
        	\end{equation*}
    	
        \item[$(s_5)$]: $f$ fixes $x,y$ with a preperiodic point $z$
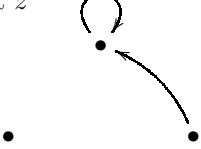
        	\begin{equation*}
        		\xymatrix{
        		&\bullet \ar@(ul,ur)& \\
        		\bullet \ar@(dl,dr) && \bullet \ar@/_/[lu]\\
        			}
        	\end{equation*}
\vspace*{5pt}
        \item[$(s_6)$]: $f$ fixes $x$ with two distinct preperiodic points $y,z$
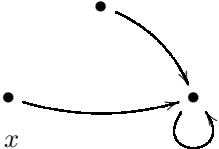
        	\begin{equation*}
        		\xymatrix{
        		&\bullet \ar@/^/ [rd]& \\
        		\bullet \ar@/_/[rr]&& \bullet \ar@(dl,dr)\\
        			}
        	\end{equation*}	

    	\item[$(s_7)$]: $f$ fixes $x$ with $f(z)=y$ and $f(y)=x$
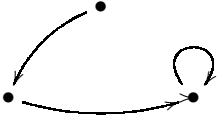
        	\begin{equation*}
        		\xymatrix{
        		&\bullet \ar@/_/ [ld]& \\
        		\bullet \ar@/_/[rr]&& \bullet \ar\ar@(ul,ur)\\
        			}
        	\end{equation*}
    \end{enumerate}

    Label the three fixed points of $S$ as $x = (x_0,x_1,x_2)$, $y = (y_0,y_1,y_2)$, and $z = (z_0,z_1,z_2)$. Then
    $S$ can be written in the form
    \begin{equation*}
        S = \left( C_1 | C_2 | C_3 \right),
    \end{equation*}
    where $C_1,C_2,C_3$ are the following column vectors.
    \begin{align}
        C_1 = \begin{pmatrix} (x_0y_2z_1 - x_0y_1z_2) \zeta_1 + (-x_2y_0z_1 + x_1y_0z_2) \zeta_2 + (x_2y_1z_0 - x_1y_2z_0)\label{eq:matrix} \\
        (x_1y_2z_1 - x_1y_1z_2) \zeta_1 + (-x_2y_1z_1 + x_1y_1z_2) \zeta_2 + (x_2y_1z_1 - x_1y_2z_1) \\
        (x_2y_2z_1 - x_2y_1z_2) \zeta_1 + (-x_2y_2z_1 + x_1y_2z_2) \zeta_2 + (x_2y_1z_2 - x_1y_2z_2) \\ \end{pmatrix}\\
        C_2 = \begin{pmatrix} (-x_0y_2z_0 + x_0y_0z_2) \zeta_1 + (x_2y_0z_0 - x_0y_0z_2) \zeta_2 + (-x_2y_0z_0 + x_0y_2z_0) \notag\\
        (-x_1y_2z_0 + x_1y_0z_2) \zeta_1 + (x_2y_1z_0 - x_0y_1z_2) \zeta_2 + (-x_2y_0z_1 + x_0y_2z_1)\\
        (-x_2y_2z_0 + x_2y_0z_2) \zeta_1 + (x_2y_2z_0 - x_0y_2z_2) \zeta_2 + (-x_2y_0z_2 + x_0y_2z_2)\\ \end{pmatrix}\\
        C_3 = \begin{pmatrix}    (x_0y_1z_0 - x_0y_0z_1) \zeta_1 + (-x_1y_0z_0 + x_0y_0z_1) \zeta_2 + (x_1y_0z_0 - x_0y_1z_0) \\
        (x_1y_1z_0 - x_1y_0z_1) \zeta_1 + (-x_1y_1z_0 + x_0y_1z_1) \zeta_2 + (x_1y_0z_1 - x_0y_1z_1) \\
        (x_2y_1z_0 - x_2y_0z_1) \zeta_1 + (-x_1y_2z_0 + x_0y_2z_1) \zeta_2 + (x_1y_0z_2 - x_0y_1z_2) \\ \end{pmatrix}, \notag
    \end{align}
    where $\zeta_1,\zeta_2$ are some power of a root of unity, not necessarily distinct. \\

    However, the points $x,y,z$ may be defined over an extension field. The next two lemmas determine which fields they can be defined over.

    \begin{lem}\label{sec:lem1}
        Let $S \in \Aut(f)$ with exactly two distinct fixed points. Let $z_1,z_2 \in \bar{K}$ be the fixed points of $S$. Fix $U \in \PGL_3(\bar{K})$ such that $U(z_1) = (1,0,0)$ and $U(z_2) = (0, 1, 0)$. Then
        \begin{equation*}
            S = U^{-1} \begin{pmatrix} \zeta_1&0&0\\0&\zeta_2&0\\0&0&1 \end{pmatrix} U
        \end{equation*}
        defines a nontrivial order $n$ element of $\PGL_3(K)$ if and only if $K(\zeta_1,\zeta_2) = K(z_1,z_2)$, where $\zeta_1, \zeta_2$ are some power of an $n$-th root of unity.
    \end{lem}

    \begin{proof}
        Let $F_1 = K(\zeta_1,\zeta_2)$. As $\zeta_1$ and $\zeta_2$ are the eigenvalues of $S$, we know that they are the roots of its characteristic polynomial and, therefore, the field $F_1$ is Galois. Now take the Galois closure of the extension $F_2 = F_1(z_1,z_2)$ and call it $L$. Note that $F_1=F_2$ if and only if $\Gal(L/F_1)$ is the trivial group. We proceed to show this by contradiction. Our matrix $S$ is rational, so it is fixed by elements of $\Gal(L/F_1)$. Therefore, the elements of $\Gal(L/F_1)$ must permute the fixed points of $S$. Suppose that $\sigma \in \Gal(F_2/F_1)$ is a nontrivial element. Since there are only two fixed points of $S$, any nontrivial permutation must swap the two fixed points so that $\sigma(z_1) = z_2$. We first examine the $(1,1)$ entry of $S$ (from (\ref{eq:matrix})) where $z_1 = (x_0,x_1,x_2)$ and $z_2 = (y_0,y_1,y_2)$. The entry is
        \begin{equation*}
            (x_0y_2z_1 - x_0y_1z_2) \zeta_1 + (-x_2y_0z_1 + x_1y_0z_2) \zeta_2 + (x_2y_1z_0 - x_1y_2z_0),
        \end{equation*}
        which, after applying $\sigma$ to it, becomes
        \begin{equation*}
            (y_0x_2z_1 - y_0x_1z_2) \zeta_1 + (-y_2x_0z_1 + y_1x_0z_2) \zeta_2 + (y_2x_1z_0 - y_1x_2z_0).
        \end{equation*}
        We see that $\sigma(\zeta_1) = \zeta_2$. Examining the other $8$ entries in the matrix, we arrive at the same result so that $\sigma(\zeta_1) = \zeta_2$ for the entire matrix. Therefore, $\sigma$ does not fix the base field and cannot be in $\Gal(L/F_1)$.
    \end{proof}

    \begin{lem}\label{sec:lem2}
        Let $S \in \Aut(f)$ with three distinct fixed points. Let $z_1,z_2,z_3 \in \bar{K}$ be the fixed points of $S$. Fix $U \in \PGL_3(\bar{K})$ such that $U(z_1) = (1,0,0)$,  $U(z_2) = (0, 1, 0)$ and $U(z_3) = (0,0,1)$. Then
        \begin{equation*}
            S = U^{-1} \begin{pmatrix} \zeta_1&0&0\\0&\zeta_2&0\\0&0&1 \end{pmatrix} U
        \end{equation*}
        defines a nontrivial order $n$ element of $\PGL_3(K)$ if and only if $K(\zeta_1,\zeta_2) = K(z_1,z_2,z_3)$, where $\zeta_1, \zeta_2$ are some power of an $n$-th root of unity.
    \end{lem}

\begin{proof}
    Let $F_1 = K(\zeta_1,\zeta_2)$ be the Galois field of the eigenvalues as in Lemma \ref{sec:lem1} and consider the Galois closure of the extension $F_2 = F_1(z_1,z_2,z_3)$, which we denote by $L$. We again show by contradiction that the only possible elements of $\Gal(L/F_1)$ are trivial, so that $F_1=F_2$.

    The matrix $S$ is rational, so it is fixed by any element of $\Gal(L/F_1)$. Therefore, the elements of $Gal(L/F_1)$ permute the fixed points. Suppose that $\sigma \in \Gal(K/L)$ is nontrivial, i.e., $\sigma(z_i) \in \{z_j,z_k\}$, where $i \ne j,k$. As in Lemma \ref{sec:lem1}, we examine how $\sigma$ acts on the matrix entries from (\ref{eq:matrix}). Without loss of generality there are two possibilities: $\sigma(z_1) = z_2$ or  $\sigma(z_1) = z_3$; therefore, $\sigma(\zeta_1) = \zeta_2$ or $\sigma(\zeta_1) = 1$. In particular, $\sigma$ does not fix the base field so cannot be in $\Gal(L/F_1)$.
\end{proof}

\subsection{Description of Algorithm}
    With the results of the previous section, we are now ready to describe the algorithm.

    Lemma \ref{sec:lem1} and Lemma \ref{sec:lem2} describe the general form of the automorphism. In particular, the possible candidates for automorphisms come from the appropriate roots of unity and the fixed points, 2-periodic points, and 3-periodic points of $f$ arranged in the cases $s_1,
    \ldots, s_7$ along with any needed pre-images.

    An upper bound on the order of the roots of unity is provided by the bound on the size of the automorphism group from Theorem \ref{thm_degree_bound}. Note that for $\P^1$, \cite{FMV} was able to use the stronger divisibility condition $n \mid d(d+1)(d-1)$ instead of just an upper bound on the size of the group. For morphisms of $\P^2$, this result is not true. For example, $C_5$, the cyclic group with $5$ elements can be the automorphism group for a functions with any degree $d \ge 5$. Since the fixed points of $S$ are the eigenvectors of the $S$, we can further limit the roots of unity to those contained in a degree $3$ extension of the base field.

    From Lemma \ref{sec:lem1} and Lemma \ref{sec:lem2}, these roots of unity determine (a finite number) of field extensions where the 1,2,3-perioidic points are rational. The algorithm computes these points and arranges them according to the combinatorics $s_1,\ldots, s_7$.

    Given the combinatorial cases, we loop through triples in each case and construct the candidate $S$ from Lemma \ref{sec:lem1} and Lemma \ref{sec:lem2}. We then test each candidate $S$ to see if it is an actual automorphism.

    The proofs of the previous section demonstrate that all elements of the automorphism group will be detected with this method. The bound on the size of the automorphism group and, hence, the roots of unity ensure that the algorithm will terminate after testing finitely many candidate automorphisms. Pseudocode for the algorithm is given in Algorithm \ref{fig:alg2}. The special case where the base field is $\Q$, where there are only two possible field extensions, is implemented in the Sage computer algebra system \cite{sage} and is described in Algorithm \ref{fig:alg1}.


\begin{algorithm}
	\caption{An algorithm to compute automorphism groups over $\P^2$ for number fields of characteristic 0 \label{fig:alg2}}
	\begin{algorithmic}
		\REQUIRE A morphism $f:\P^2 \rightarrow \P^2$ defined over $K$
		\ENSURE The automorphism group for $f$
		\FOR {j: 1 $\rightarrow 6\deg(f)^6$}
			\STATE Let $C(X) = X^j -1$
			\STATE Create a list of factors of C(X) for up to degree 3 extension of $K$
			\STATE Let $a,b$ be any positive integers such that $gcd(a,j)= gcd(b,j)=1$
			\STATE Compute the $3$-periodic points, $2$-perioidc points, fixed points over $K(\zeta_j^a, \zeta_j^b)$
			\FOR {i: 1 $\rightarrow$ 7}
				\STATE Construct $s_i$
				\FOR{ $(x,y,z) \in s_i$}
					\STATE Construct U s.t. $U(x)=(1,0,0)$ $U(y) = (0,1,0)$  and $U(z) = (0,0,1)$
					\STATE Set $S(x) = U^{-1} \begin{pmatrix} \zeta_j^a&0&0\\0&\zeta_j^b&0 \\0&0&1  \end{pmatrix} U$
					\STATE Test if $S \in \Aut(f)$
				\ENDFOR
			\ENDFOR
		\ENDFOR
	\end{algorithmic}
\end{algorithm}


\begin{algorithm}
	\caption{An algorithm to compute automorphism groups over $\P^2$ for $\Q$ \label{fig:alg1}}
	\begin{algorithmic}
		\REQUIRE A morphism $f:\P^2 \rightarrow \P^2$ of degree $\ge 2$ defined over $\Q$
		\ENSURE The automorphism group for $f$
		\STATE Compute the $3$-periodic points of $f$ over $F$, $F(\zeta_4)$ and $F(\zeta_6)$
		\STATE Compute the $2$-periodic points of $f$ over $F$, $F(\zeta_4)$ and $F(\zeta_6)$
		\STATE Compute the fixed points of $f$ over $F$, $F(\zeta_4)$ and $F(\zeta_6)$
		\STATE Construct the sets $s_1 \rightarrow s_2$ over $F$, $F(\zeta_4)$ and $F(\zeta_6)$
		\FOR {i: 1 $\rightarrow$ 7}
			\FOR{ $(x,y,z) \in s_i$ over $F$ }
				\STATE Construct U s.t. $U(x)=(1,0,0)$ $U(y) = (0,1,0)$  and $U(z) = (0,0,1)$
				\STATE Let $a,b$ be any positive integers such that$gcd(a,n)= gcd(b,n)=1$
				\STATE Set $S(x) = U^{-1} \begin{pmatrix} (-1)^a&0&0\\0&(-1)^b&0 \\0&0&1  \end{pmatrix} U$
				\STATE Test if $S \in \Aut(f)$
			\ENDFOR
			\FOR{ $(x,y,z) \in s_i$ over $F(\zeta_4)$ }
				\STATE Construct U s.t. $U(x)=(1,0,0)$ $U(y) = (0,1,0)$  and $U(z) = (0,0,1)$
				\STATE Let $a,b$ be any positive integers such that$gcd(a,n)= gcd(b,n)=1$
				\STATE Set $S(x) = U^{-1} \begin{pmatrix} (i)^a&0&0\\0&(i)^b&0 \\0&0&1  \end{pmatrix} U$
				\STATE Test if $S \in \Aut(f)$
			\ENDFOR
			\FOR{ $(x,y,z) \in s_i$ over $F(\zeta_6)$ }
				\STATE Construct U s.t. $U(x)=(1,0,0)$ $U(y) = (0,1,0)$  and $U(z) = (0,0,1)$
				\STATE Let $a,b$ be any integers such that$gcd(a,n)= gcd(b,n)=1$
				\STATE Set $S(x) = U^{-1} \begin{pmatrix} (\frac{1 + i\sqrt{3}}{2})^a&0&0\\0&(\frac{1 + i\sqrt{3}}{2})^b&0 \\0&0&1  \end{pmatrix} U$
				\STATE Test if $S \in \Aut(f)$
			\ENDFOR
		\ENDFOR
	\end{algorithmic}
\end{algorithm}

\section{Algorithm Examples} \label{sect.alg.exmp}

\subsection{Examples on $\P^2$ using invariant theory} \label{sect.exmp_dim2}
    \begin{exmp}
        As in dimension 1 (Example \ref{exmp.octahedral}), we give an example of using invariant theory to find endomorphisms of $\P^2$ with nontrivial stabilizer group. This example uses the equivariant Molien series and Reynolds operator.

        Consider the group of order 216 generated by
        \begin{equation*}
            \left\langle \begin{pmatrix} 1&0&0\\0&\zeta_3&0\\0&0&\zeta_3^2 \end{pmatrix},
            \begin{pmatrix} 0&1&0\\0&0&1\\1&0&0 \end{pmatrix},
            \frac{1}{\zeta3-\zeta_3^2}
            \begin{pmatrix} 1&1&1\\1&\zeta_3&\zeta_3^2\\1&\zeta_3^2&\zeta_3 \end{pmatrix},
            \begin{pmatrix} \eta&0&0\\0&\eta&0\\0&0& \eta \zeta_3 \end{pmatrix}\right \rangle 
        \end{equation*}
        where $\eta^3 = \zeta_3^2$.

        We find three distinct classes of linear characters and compute the corresponding equivariant Molien series.
        \begin{align*}
            \chi_1&: t + 2t^{10} + t^{13} + t^{16} + 6t^{19} + 2t^{22} + 4t^{25} + 9t^{28} + 6t^{31} + O(t^{32})\\
            \chi_2&: 2t^7 + t^{13} + 4t^{16} + 2t^{19} + 2t^{22} + 9t^{25} + 4t^{28} + 6t^{31} + O(t^{32})\\
            \chi_3&: t^4 + 4t^{13} + t^{16} + 2t^{19} + 6t^{22} + 4t^{25} + 4t^{28} + 12t^{31} + O(t^{32}).
        \end{align*}
        From these we know the degrees of the maps with $\Gamma \subseteq \mathcal{A}_{f}$ and can search with the equivariant Reynolds operator.
        For example, in lowest degree for each character we find the maps:
        \begin{align*}
            \chi_1 &\colon [-2x^7y^3 + 2x^7z^3 - 7x^4y^6 + 7x^4z^6 + xy^9 - 14xy^6z^3 + 14xy^3z^6 - xz^9\\
            &\qquad -x^9y + 7x^6y^4 + 14x^6yz^3 + 2x^3y^7 - 14x^3yz^6 - 2y^7z^3 - 7y^4z^6 + yz^9\\
            &\qquad x^9z - 14x^6y^3z - 7x^6z^4 + 14x^3y^6z - 2x^3z^7 - y^9z + 7y^6z^4 + 2y^3z^7].\\
            \chi_2 &\colon [ 8x^7 - 35x^4y^3 - 35x^4z^3 - 7xy^6 - 140xy^3z^3 - 7xz^6\\
            &\qquad -7x^6y - 35x^3y^4 - 140x^3yz^3 + 8y^7 - 35y^4z^3 - 7yz^6\\
            &\qquad-7x^6z - 140x^3y^3z - 35x^3z^4 - 7y^6z - 35y^3z^4 + 8z^7].\\
            \chi_3 &\colon [ xy^3 - xz^3, -x^3y + yz^3, x^3z - y^3z].
        \end{align*}

\begin{code}
{\footnotesize \begin{verbatim}
K<z3>:=CyclotomicField(3);
R<x>:=PolynomialRing(K);
f:=x^3-z3^2;
L<e>:=ext<K|f>;
R<x,y,z> := PolynomialRing(L,3);
X:=Matrix(R,3,1,[x,y,z]);

H := MatrixGroup< 3, L| [1,0,0, 0,z3,0, 0,0,z3^2],[0,1,0, 0,0,1, 1,0,0],[1/(z3-z3^2),1/(z3-z3^2),1/(z3-z3^2), 1/(z3-z3^2),z3/(z3-z3^2),z3^2/(z3-z3^2), 1/(z3-z3^2),z3^2/(z3-z3^2),z3/(z3-z3^2)], [e,0,0, 0,e,0, 0,0,e*z3] >;


//Compute the equivariant Molien series.
//Only CT[1], CT[2], CT[3] are linear.

CT:=CharacterTable(H);
xi:=CT[1];
S:=Inverse(NumberingMap(H));
R<t>:=PolynomialRing(L);
EMol:=0;
for i:=1 to Order(H) do
  EMol := EMol + (R!L!K!xi(S(i)))*Trace(Matrix(R,S(i)^(-1)))/(Determinant(DiagonalMatrix([1,1,1])- t*Matrix(R,S(i))));
end for;
EMol:=EMol/Order(H);
PS<t>:=PowerSeriesRing(L);
PS!EMol;

//Apply the equivariant Reynolds operator
R<x,y,z>:=PolynomialRing(L,3);
X:=Matrix(R,3,1,[x,y,z]);
M:=MonomialsOfDegree(R,4);
Auts:={};

for xxi:=1 to #M do
  for yi:=1 to #M do
    for zi:=1 to #M do
        T:=Matrix(3,1,[0,0,0]);
        F:=Matrix(R,3,1,[M[xxi],M[yi],M[zi]]);
      for i:= 1 to #H do
        t:=Matrix(R,S(i))*X;
        G:=Matrix(R,3,1,[Evaluate(F[1,1],[t[1,1],t[2,1],t[3,1]]),Evaluate(F[2,1],[t[1,1],t[2,1],t[3,1]]),Evaluate(F[3,1],[t[1,1],t[2,1],t[3,1]])]);
        C:=Matrix(R,S(i))^(-1)*G;
        T:=T+R!L!K!xi(S(i))*C;
      end for;
      g:=GCD([T[1,1],T[2,1],T[3,1]]);
     if g ne 0 then
        T[1,1]:=T[1,1]/g;
        T[2,1]:=T[2,1]/g;
        T[3,1]:=T[3,1]/g;
      end if;
      Auts:=Include(Auts, T);
    end for;
  end for;
end for;
Auts;
\end{verbatim}}
\end{code}
    \end{exmp}

\begin{exmp}
The group of order 60 generated by
    \[ \left \langle \begin{pmatrix} 0&1&0\\0&0&1\\1&0&0 \end{pmatrix},
       \begin{pmatrix} 1&0&0\\0&-1&0\\0&0&-1 \end{pmatrix},
       \begin{pmatrix} -1&\mu_2&\mu_1\\ \mu_2&\mu_1&-1\\ \mu_1&-1&\mu_2 \end{pmatrix} \right \rangle,\]
    where $\mu_1=\frac{1}{2}(-1+\sqrt{5})$ and $\mu_2 = \frac{1}{2}(-1-\sqrt{5})$.

    In this example, we use invariant polynomials to construct the invariant 2-form. For the group (H), we compute the three fundamental invariants $F_2,F_6,F_{10}$, which are degrees $2,6,10$ respectively.
    \begin{align*}
        F_2&=x^2 + y^2 + z^2\\
        F_6&=x^6 + 1/14(-3\sqrt{5} + 39)x^4y^2 + 1/14(3\sqrt{5} + 39)x^4z^2 + 1/14(3\sqrt{5} + 39)x^2y^4 + 54/7x^2y^2z^2\\
        &+ 1/14(-3\sqrt{5} + 39)x^2z^4 + y^6 + 1/14(-3\sqrt{5} + 39)y^4z^2 +1/14(3\sqrt{5} + 39)y^2z^4 + z^6\\
        F_{10}&=x^{10} + 1/38(-27\sqrt{5} + 153)x^8y^2 + 1/38(27\sqrt{5} + 153)x^8z^2 + 1/19(-21\sqrt{5} + 147)x^6y^4 + 504/19x^6y^2z^2\\
        &+ 1/19(21\sqrt{5} + 147)x^6z^4 + 1/19(21\sqrt{5} +147)x^4y^6 + 630/19x^4y^4z^2 + 630/19x^4y^2z^4\\
        &+ 1/19(-21\sqrt{5} + 147)x^4z^6 + 1/38(27\sqrt{5} + 153)x^2y^8 + 504/19x^2y^6z^2 + 630/19x^2y^4z^4 + 504/19x^2y^2z^6\\
        &+ 1/38(-27\sqrt{5} + 153)x^2z^8 + y^{10} + 1/38(-27\sqrt{5} + 153)y^8z^2 + 1/19(-21\sqrt{5} + 147)y^6z^4\\
        &+ 1/19(21\sqrt{5} + 147)y^4z^6 + 1/38(27\sqrt{5} + 153)y^2z^8 + z^{10}
    \end{align*}
    We then compute the invariant $2$-form $dF_2 \wedge dF_6$ as $\nabla F_2 \times \nabla F_6$ to get the map
{\small
    \begin{align*}
        f&=[12 \sqrt{5} x^{4} y z + \left(-12 \sqrt{5} + 60\right) x^{2} y^{3} z + \left(-6 \sqrt{5} -
        6\right) y^{5} z + \left(-12 \sqrt{5} - 60\right) x^{2} y z^{3} + 24 \sqrt{5} y^{3}
        z^{3} + \left(-6 \sqrt{5} + 6\right) y z^{5}\\
        &: \left(-6 \sqrt{5} + 6\right) x^{5} z + \left(-12 \sqrt{5} - 60\right) x^{3} y^{2} z +
        12 \sqrt{5} x y^{4} z + 24 \sqrt{5} x^{3} z^{3} + \left(-12 \sqrt{5} + 60\right) x y^{2}
        z^{3} + \left(-6 \sqrt{5} - 6\right) x z^{5}\\
        &: \left(-6 \sqrt{5} - 6\right) x^{5} y + 24 \sqrt{5} x^{3} y^{3} + \left(-6 \sqrt{5} + 6\right)
        x y^{5} + \left(-12 \sqrt{5} + 60\right) x^{3} y z^{2} + \left(-12 \sqrt{5} -
        60\right) x y^{3} z^{2} + 12 \sqrt{5} x y z^{4}].
    \end{align*}
}
    However, notice that this method fails to construct the equivariant of degree $5$ which we expect from the exterior Molien series:
    \begin{equation*}
        s^3 + (t + t^5 + t^6 + t^9 + t^{10} + t^{14} - t^{16} + O(t^{20}))s^2 + (t + t^5 + t^6 + t^9 + t^{10} + t^{14} - t^{16} + O(t^{20}))s + 1
    \end{equation*}
    (after dividing out by the Molien series). Using the equivariant Reynolds operator, we can find one of degree $5$.
    \begin{align*}
        f=&[70x^5 + (-10\sqrt{5} + 130)x^3y^2 + (5\sqrt{5} + 65)xy^4 + (10\sqrt{5} +130)x^3z^2 + 180xy^2z^2 + (-5\sqrt{5} + 65)xz^4\\
        &: (-5\sqrt{5} + 65)x^4y +(10\sqrt{5} + 130)x^2y^3 + 70y^5 + 180x^2yz^2 + (-10\sqrt{5} + 130)y^3z^2 +(5\sqrt{5} + 65)yz^4\\
        &: (5\sqrt{5} + 65)x^4z + 180x^2y^2z + (-5\sqrt{5} + 65)y^4z+ (-10\sqrt{5} + 130)x^2z^3 + (10\sqrt{5} + 130)y^2z^3 + 70z^5)]
    \end{align*}
\end{exmp}

\begin{code}
\begin{verbatim}
R<x>:=PolynomialRing(Rationals());
f:=x^2-5;
L<a>:=NumberField(f);
mu1:=1/2*(-1+a);
mu2:=1/2*(-1-a);
R<x,y,z> := PolynomialRing(L,3);
X:=Matrix(R,3,1,[x,y,z]);

H := MatrixGroup< 3, L| [0,1,0, 0,0,1, 1,0,0], [1,0,0, 0,-1,0, 0,0,-1], [-1/2,mu2/2,mu1/2, mu2/2,mu1/2,-1/2, mu1/2,-1/2,mu2/2] >;

CT:=CharacterTable(H); //only the trivial character is linear

//compute the Molien Series
CT:=CharacterTable(H);
xi:=CT[1];
S:=Inverse(NumberingMap(H));
R<t>:=PolynomialRing(L);
Mol:=0;
for i:=1 to Order(H) do
  Mol := Mol + (L!xi(S(i)))*Determinant(S(i))/(Determinant(DiagonalMatrix([1,1,1])- t*Matrix(R,S(i))));
end for;
Mol:=Mol/Order(H);
Mol;
PS<t>:=PowerSeriesRing(L);
PS!Mol;
//1 + t^2 + t^4 + 2*t^6 + 2*t^8 + 3*t^10 + 4*t^12 + 4*t^14 + t^15 + 5*t^16 + t^17 + 6*t^18 + t^19 + O(t^20)

//Molien Series for 2-forms
Mol:=0;
Mol0:=0;
for j:=1 to 1 do
  xi:=CT[j];
  S:=Inverse(NumberingMap(H));
  PS<t>:=PowerSeriesRing(L);
  PR<s>:=PolynomialRing(PS);
  Mol:=0;
  Mol0:=0;
  for i:=1 to Order(H) do
    Mol := Mol + (PR!xi(S(i))*(Determinant(DiagonalMatrix([1,1,1]) + s*Matrix(L,S(i)))))/(Determinant(DiagonalMatrix([1,1,1]) - t*Matrix(L,S(i))));
    Mol0 := Mol0 + (PS!xi(S(i)))/(Determinant(DiagonalMatrix([1,1,1]) - t*Matrix(L,S(i))));
  end for;
  Mol:=Mol/Order(H);
  Mol0:=Mol0/Order(H);
end for;
Mol;
Mol/(Mol0);

//(1 + O(t^20))*s^3 + (t + t^5 + t^6 + t^9 + t^10 + t^14 - t^16 + O(t^20))*s^2 + (t + t^5 + t^6 + t^9 + t^10 + t^14 - t^16 + O(t^20))*s + 1 + O(t^20)

Inv:=PrimaryInvariants(InvariantRing(H));
f0:=R!Inv[1];
f1:=R!Inv[2];
f2:=R!Inv[3];
df0:=[Derivative(f0,R.1), Derivative(f0,R.2), Derivative(f0,R.3)];
df1:=[Derivative(f1,R.1), Derivative(f1,R.2), Derivative(f1,R.3)];
df2:=[Derivative(f2,R.1), Derivative(f2,R.2), Derivative(f2,R.3)];

u:=df0;
v:=df1;
F:=Matrix(R,3,1,[u[2]*v[3] - u[3]*v[2], u[3]*v[1]-u[1]*v[3], u[1]*v[2] - u[2]*v[1]]);

f0:=70*x^5 + (-10*a + 130)*x^3*y^2 + (5*a + 65)*x*y^4 + (10*a +130)*x^3*z^2 + 180*x*y^2*z^2 + (-5*a + 65)*x*z^4;
f1:=(-5*a + 65)*x^4*y +(10*a + 130)*x^2*y^3 + 70*y^5 + 180*x^2*y*z^2 + (-10*a + 130)*y^3*z^2 +(5*a + 65)*y*z^4;
f2:=(5*a + 65)*x^4*z + 180*x^2*y^2*z + (-5*a + 65)*y^4*z+ (-10*a + 130)*x^2*z^3 + (10*a + 130)*y^2*z^3 + 70*z^5;
F:=Matrix(R,3,1,[f0,f1,f2]);

Auts:={};
X:=Matrix(R,3,1,[x,y,z]);
S:=Inverse(NumberingMap(H));
for i:=1 to Order(H) do
    t:=Matrix(R,S(i))*X;
    G:=Matrix(R,3,1,[Evaluate(F[1,1],[t[1,1],t[2,1],t[3,1]]),Evaluate(F[2,1],[t[1,1],t[2,1],t[3,1]]),Evaluate(F[3,1],[t[1,1],t[2,1],t[3,1]])]);
    C:=Matrix(R,S(i))^(-1)*G;
    if C[1,1]*F[2,1] eq C[2,1]*F[1,1] and C[1,1]*F[3,1] eq C[3,1]*F[1,1] and C[2,1]*F[3,1] eq C[3,1]*F[2,1] then
        Auts:=Include(Auts, S(i));
    end if;
end for;
#Auts;

\end{verbatim}
\end{code}

\begin{code}
\begin{verbatim}
//check
f0:=70*x^5 + (-10*a + 130)*x^3*y^2 + (5*a + 65)*x*y^4 + (10*a +130)*x^3*z^2 + 180*x*y^2*z^2 + (-5*a + 65)*x*z^4;
f1:=(-5*a + 65)*x^4*y +(10*a + 130)*x^2*y^3 + 70*y^5 + 180*x^2*y*z^2 + (-10*a + 130)*y^3*z^2 +(5*a + 65)*y*z^4;
f2:=(5*a + 65)*x^4*z + 180*x^2*y^2*z + (-5*a + 65)*y^4*z+ (-10*a + 130)*x^2*z^3 + (10*a + 130)*y^2*z^3 + 70*z^5;
F:=Matrix(R,3,1,[f0,f1,f2]);

Auts:={};
X:=Matrix(R,3,1,[x,y,z]);
S:=Inverse(NumberingMap(H));
for i:=1 to Order(H) do
    t:=Matrix(R,S(i))*X;
    G:=Matrix(R,3,1,[Evaluate(F[1,1],[t[1,1],t[2,1],t[3,1]]),Evaluate(F[2,1],[t[1,1],t[2,1],t[3,1]]),Evaluate(F[3,1],[t[1,1],t[2,1],t[3,1]])]);
    C:=Matrix(R,S(i))^(-1)*G;
    if C[1,1]*F[2,1] eq C[2,1]*F[1,1] and C[1,1]*F[3,1] eq C[3,1]*F[1,1] and C[2,1]*F[3,1] eq C[3,1]*F[2,1] then
        Auts:=Include(Auts, S(i));
    end if;
end for;
#Auts;
\end{verbatim}
\end{code}

\begin{code}
\begin{verbatim}
R.<x>=QQ[]
K.<a>=NumberField(x^2-5)
P.<x,y,z>=ProjectiveSpace(K,2)
H=End(P)
f0=(70*x^5 + (-10*a + 130)*x^3*y^2 + (5*a + 65)*x*y^4 + (10*a +130)*x^3*z^2 + 180*x*y^2*z^2 + (-5*a + 65)*x*z^4;
f1=(-5*a + 65)*x^4*y +(10*a + 130)*x^2*y^3 + 70*y^5 + 180*x^2*y*z^2 + (-10*a + 130)*y^3*z^2 +(5*a + 65)*y*z^4;
f2=(5*a + 65)*x^4*z + 180*x^2*y^2*z + (-5*a + 65)*y^4*z+ (-10*a + 130)*x^2*z^3 + (10*a + 130)*y^2*z^3 + 70*z^5);

f=H([f0,f1,f2])
m1=matrix(3,[0,1,0, 0,0,1, 1,0,0])
m2 =matrix(3,[1,0,0, 0,-1,0, 0,0,-1])
f.conjugate(m1)==f,f.conjugate(m2)==f
\end{verbatim}
\end{code}

\begin{code}
\begin{verbatim}
R.<x>=QQ[]
K.<a>=NumberField(x^2-5)
P.<x,y,z>=ProjectiveSpace(K,2)
H=End(P)
f0=12/7*a*x^4*y*z + 1/7*(-12*a + 60)*x^2*y^3*z + 1/7*(-12*a - 60)*x^2*y*z^3 + 1/7*(-6*a - 6)*y^5*z +24/7*a*y^3*z^3 + 1/7*(-6*a + 6)*y*z^5
f1=1/7*(-6*a + 6)*x^5*z + 1/7*(-12*a - 60)*x^3*y^2*z + 24/7*a*x^3*z^3 + 12/7*a*x*y^4*z + 1/7*(-12*a +60)*x*y^2*z^3 + 1/7*(-6*a - 6)*x*z^5
f2=1/7*(-6*a - 6)*x^5*y + 24/7*a*x^3*y^3 + 1/7*(-12*a + 60)*x^3*y*z^2 + 1/7*(-6*a + 6)*x*y^5 + 1/7*(-12*a - 60)*x*y^3*z^2 + 12/7*a*x*y*z^4
f=H([f0,f1,f2])

m1=matrix(3,[0,1,0, 0,0,1, 1,0,0])
m2 =matrix(3,[1,0,0, 0,-1,0, 0,0,-1])
f.conjugate(m1)==f,f.conjugate(m2)==f
\end{verbatim}
\end{code}

\subsection{Method of Fixed Points Examples} \label{sect.FP.exmp}
    The algorithm for base field $\Q$ is implemented in Sage \cite{sage}, and we discuss examples and that implementation in this section.
    Since Sage contains a fast algorithm to compute periodic points over $\Q$ \cite{HutzC} but must rely on slower more naive computations for general number fields, the slowest portion of the algorithm in practice is actually determining the periodic points. In particular, the differences between the actual runtime for the examples cannot be attributed only to the degree of the morphisms, but instead to the dynamical structure of the function, i.e. how many rational periodic points there are over the needed fields. Unfortunately, the problem of rational periodic point structure is not yet well understood. Morton-Silverman have conjectured a uniform upper bound on the number of preperiodic points of a morphism on $\P^N$ \cite{MortSilverman} depending on the degree of the map and the degree of the field, but this is still an open problem.

    The algorithm was run with Sage 6.7 on an Ubuntu $14.04.2$ VirtualBox on an Acer Aspire V3 with a dedicated $8$Gb of RAM and an Intel i-$7$ $2.5$GHz processor. Figure \ref{fig:table1} gives a few examples with runtime listed. Note that the runtime with the asterisk is described in Example \ref{exmp2} and was run without computing the 3-periodic points since they are not needed.

    \begin{figure}
    \caption{Experimental Run Times}
    \label{fig:table1}
        \begin{center}
             \begin{tabular}{||c c c c||}
             \hline
             Morphism & RunTime &Group Classification& Size \\ [0.5ex]
             \hline\hline
             $(x^3:y^3:z^3)$ & 49.65s & octahedral & 24 \\
             \hline
            $(2x^3 + xy^2: 2y^3 +yz^2:x^2z +2z^3)$ & 17.79s* & tetrahedral & 12 \\
            \hline
            $(x^2:y^2:z^2)$ & 16s & dihedral& 6\\
            \hline
            $(x^3:x^2y+y^3:z^3)$ & 33.68s & $C_2 \times C_2$ & 4\\
             \hline
             $(x^3 + xy^2: yx^2 +2y^3:z^3)$ & 54s & $C_2 \times C_2$ & 4\\
             \hline
            $(x^2 +y^2:y^2+z^2:z^2)$ & 2.46s& trivial & 1 \\
            \hline
            $(x^2 + y^2: y^2:z^2)$  &11.09s& trivial & 1 \\
            \hline
            $(x^3 + 6987y^3:y^3:z^3)$  & 16.34s & trivial & 1 \\
            \hline
            $(x^4 + y^4: y^4:z^4)$  &19.15s& trivial & 1 \\
            \hline
            \end{tabular}
        \end{center}
    \end{figure}

\begin{exmp}\label{exmp2}
    For $f(x,y,z) =(2x^3 + xy^2: 2y^3 +yz^2:x^2z +2z^3)$, $\Aut(f)$ is given by the dihedral group. Computing the rational 3-periodic points was infeasible; however, we were able to prove that there are no rational 3-periodic points for the number fields needed in the algorithm.

    Reducing modulo $23$, which is prime over both fields, we use the description of the cycle length of rational points compared to the cycles length modulo primes of good reduction described by the second author \cite{Hutz2} to see there are no $3$-cycles. In particular, we examine the cycles over $\F_{23^2}$ to find the cycles defined over the residue fields of the two needed quadratic fields. We can conclude that there are no 3-cycles for either field. Therefore, we can exclude the cases involving the 3-cycles and the algorithm is able to complete.
\end{exmp}

\bibliography{auto_bib}
\bibliographystyle{plain}

\end{document}